\documentclass[preprint, 12pt]{elsarticle}

\usepackage{pgfplots} 
\pgfplotsset{compat=1.18}
\usepackage{ifthen} 
\usepackage{srcltx}
\usepackage{eurosym}
\usepackage{mathtools}
\usepackage{amsmath}
\usepackage{amsfonts}
\usepackage{amssymb}
\usepackage{amsthm}
\usepackage{graphicx}
\usepackage{mathrsfs}
\usepackage{xcolor}
\usepackage{lmodern}
\usepackage{exscale}
\usepackage{latexsym}
\usepackage{esvect}
\usepackage{euler}
\usepackage{tikz}
\usepackage{comment}
\usetikzlibrary{shapes}
\usepackage{etoolbox}
\usetikzlibrary{decorations.pathreplacing}

\usepackage[overload]{empheq}
\renewcommand\emph[1]{\color{auburn}{{#1}}}
\definecolor{airforceblue}{rgb}{0.36, 0.54, 0.66}
\definecolor{auburn}{rgb}{0.43, 0.21, 0.1}
\usepackage[colorlinks,plainpages=true,pdfpagelabels,hypertexnames=true,colorlinks=true,pdfstartview=FitV,linkcolor=blue,citecolor=red,urlcolor=black,backref=page]{hyperref}
\PassOptionsToPackage{unicode}{hyperref}
\PassOptionsToPackage{naturalnames}{hyperref}
\definecolor{alizarin}{rgb}{0.82, 0.1, 0.26}

\usepackage{enumerate}
\usepackage[shortlabels]{enumitem}
\usepackage{bookmark}
\usepackage{wasysym}
\usepackage{esint}
\usepackage[titletoc,toc,page]{appendix}

\usepackage[ddmmyyyy]{datetime}
\numberwithin{equation}{section}
\everymath{\displaystyle}
\usepackage[margin=2.2cm]{geometry}

\usepackage[capitalize,nameinlink]{cleveref}
\crefname{section}{Section}{Sections}
\crefname{subsection}{Subsection}{Subsections}
\crefname{condition}{Condition}{Conditions}
\crefname{hypothesis}{Hypothesis}{Hypothesis}
\crefname{assumption}{Assumption}{Assumptions}
\crefname{lemma}{Lemma}{Lemmas}
\crefname{claim}{Claim}{Claims}
\crefname{remark}{Remark}{Remarks}

\crefformat{equation}{\textup{#2(#1)#3}}
\crefrangeformat{equation}{\textup{#3(#1)#4--#5(#2)#6}}
\crefmultiformat{equation}{\textup{#2(#1)#3}}{ and \textup{#2(#1)#3}}
{, \textup{#2(#1)#3}}{, and \textup{#2(#1)#3}}
\crefrangemultiformat{equation}{\textup{#3(#1)#4--#5(#2)#6}}%
{ and \textup{#3(#1)#4--#5(#2)#6}}{, \textup{#3(#1)#4--#5(#2)#6}}%
{, and \textup{#3(#1)#4--#5(#2)#6}}

\Crefformat{equation}{#2Equation~\textup{(#1)}#3}
\Crefrangeformat{equation}{Equations~\textup{#3(#1)#4--#5(#2)#6}}
\Crefmultiformat{equation}{Equations~\textup{#2(#1)#3}}{ and \textup{#2(#1)#3}}
{, \textup{#2(#1)#3}}{, and \textup{#2(#1)#3}}
\Crefrangemultiformat{equation}{Equations~\textup{#3(#1)#4--#5(#2)#6}}%
{ and \textup{#3(#1)#4--#5(#2)#6}}{, \textup{#3(#1)#4--#5(#2)#6}}%
{, and \textup{#3(#1)#4--#5(#2)#6}}

\crefdefaultlabelformat{#2\textup{#1}#3}
\newtheorem{theorem}{Theorem}[section]
\newtheorem{lemma}[theorem]{Lemma}
\newtheorem{corollary}[theorem]{Corollary}

\newtheorem{prop}[theorem]{Proposition}

\newtheorem{defn}[theorem]{Definition}
\newtheorem{remark}[theorem]{Remark}        
  
\numberwithin{equation}{section}

\def\YYint#1#2#3{{\setbox0=\hbox{$#1{#2#3}{\iint}$}
\vcenter{\hbox{$#2#3$}}\kern-.50\wd0}}

\def\XXint#1#2#3{{\setbox0=\hbox{$#1{#2#3}{\int}$}
\vcenter{\hbox{$#2#3$}}\kern-.50\wd0}}

\makeatletter
\def\namedlabel#1#2{\begingroup
\def\@currentlabel{#2}%
\label{#1}\endgroup
}
\makeatother
\makeatletter
\newcommand{\rmh}[1]{\mathpalette{\raisem@th{#1}}}
\newcommand{\raisem@th}[3]{\hspace*{-1pt}\raisebox{#1}{$#2#3$}}
\makeatother

\newcommand{\descref}[2]{\hyperref[#1]{\textup{\textcolor{black}{(}\textcolor{blue}{\bf #2}\textcolor{black}{)}}}}

\makeatletter
\g@addto@macro\normalsize{%
\setlength\abovedisplayskip{3pt}
\setlength\belowdisplayskip{3pt}
\setlength\abovedisplayshortskip{1pt}
\setlength\belowdisplayshortskip{3pt}
}
\makeatletter
\def\ps@pprintTitle{%
\let\@oddhead\@empty
\let\@evenhead\@empty
\def\@oddfoot{}%
\let\@evenfoot\@oddfoot}
\makeatother
\usepackage{setspace}

\newcommand\RR{\mathbb{R}}

\newcommand\ZZ{\mathbb{Z}}

\newcommand{\eps}{\varepsilon}

\DeclareMathOperator{\dv}{div}

\DeclareMathOperator{\dist}{dist}

\DeclareMathOperator{\supp}{supp}

\allowdisplaybreaks

\newcommand{\dert}{\partial_t}
\newcommand{\aeps}{a^1_{\eps}}
\newcommand{\beps}{a^2_{\eps}}
\newcommand{\ceps}{a^3_{\eps}}

\definecolor{Plum}{HTML}{89b02e}

\definecolor{Violet}{HTML}{58429B}

\definecolor{OliveGreen}{HTML}{0d8795}

\usepackage[titletoc]{appendix}

\makeatletter
\def\ps@pprintTitle{%
\let\@oddhead\@empty
\let\@evenhead\@empty
\def\@oddfoot{}%
\let\@evenfoot\@oddfoot}
\makeatother

\begin{document}
\begin{abstract}
    In this article we study the asymptotic behaviour of the solution of the three species chemical reaction-diffusion model with non-homogeneous Neumann boundary condition in a perforated domain. We investigate how the mass inflow at the microscale affects the three-species reaction–diffusion system at the macroscale. When the perforations vanish at a moderate rate, employing two-scale convergence we observe that the microscale effects are captured by a global source term in the homogenized equation, which remains a three-species reaction-diffusion system but with modified diffusion coefficients. However when the perforations vanish rapidly, the influence of the microscale mass inflow becomes negligible and the macroscopic limit remains the same reaction-diffusion system with the same diffusion coefficients.  
\end{abstract}

\begin{frontmatter}
\title{Homogenization of Three Species Reaction Diffusion Equation in Perforated Domains}
\author[sd]{Saumyajit Das}
\ead[sd]{saumyajit.math.das@gmail.com}
\address[sd]{Department of Mathematics, Harish-Chandra Research Institute, A CI of Homi Bhabha National Institute, Chhatnag Road, Jhunsi, Allahabad 211 019, India.}

\author[ks]{Kshitij Sinha}
\ead[ks]{23d0781@iitb.ac.in}
\address[ks]{Department of Mathematics, Indian Institute of Technology Bombay, Powai, Mumbai - 400076, Maharashtra, India.}
\end{frontmatter}

\section{Introduction}
Reaction-diffusion equations describe a wide range of physical phenomenon including reversible chemical kinetics, gas combustion model, diffusion of pollutants into atmosphere or electro-deposition of Nickel-Iron alloy \cite{pierre2010global,alaa2008mathematical, herrero1998global, texier2009global}. The chemical reaction-diffusion equation governs the evolution (in time) of species concentration at various spatial locations in a reversible chemical kinetics with diffusion. The general chemical  reaction-diffusion equation is associated with the following reversible chemical kinetics: 
	\[
	p_1 X_1+\cdots+p_{m}X_{m}{\xrightleftharpoons[k_1]{k_2}}q_1 X_1+\cdots+q_{m}X_{m},
	\]
where $X_i$ denotes the chemical species for $1\leq i\leq m$, $p_i$ and $q_i$ denote the stoichiometric coefficients for $1\leq i\leq m$ and $k_1$, $k_2$ represent the forward and backward reaction rates respectively. In this article, we are interested in the effect of homogenization of the three species chemical reaction diffusion equation governed by the following chemical kinetics 
\[
X_1+X_2 \leftrightharpoons X_3.
\]
The mathematical model corresponding to the above reversible chemical kinetics is given by: for $i=1,2$
\begin{equation} \label{three species reaction diffusion system}
    \left \{
    \begin{aligned}
        \partial_t a^i -d_i \Delta a^i=& a^3-a^1a^2 \qquad && \text{in}\ (0,T)\times\Omega\\
        \partial_t a^3 -d_3 \Delta a^3=& a^1a^2-a^3 \qquad && \text{in}\ (0,T)\times\Omega\\
        \nabla a^i\cdot n= \nabla& a^3\cdot n= 0 \qquad && \text{on}\ (0,T)\times\partial\Omega\\
        a^i(0,x)=& a^i_0 \qquad && \text{in}\ \Omega\\
        a^3(0,x)=& a^3_0 \qquad && \text{in}\ \Omega.
    \end{aligned}
    \right .
\end{equation}
The spatial domain is taken to be a bounded domain $\Omega\subset\mathbb{R}^3$ with $\mathrm C^{2+\nu}$ boundary with $\nu >0$. The unknowns $a^1,a^2,a^3:[0,T)\times\Omega\to\mathbb{R}$, represent the concentrations of the three species $X_1,X_2$ and $X_3$ respectively. Furthermore, here $n(x)$ denotes the outward unit normal to $\Omega$ at the point $x\in\partial\Omega$. The initial data $a^1_{0}, a^2_{0}$ and $ a^3_{0}$ are taken to be sufficiently smooth up to the closure of the domain and nonnegative. The diffusion coefficients $d_{1},d_{2} $ and $d_{3}$ are taken to be strictly positive. This particular model has attracted a significant attraction in the mathematical community, particularly in the study of the global-in-time existence of solutions \cite{pierre2010global, desvillettes2015duality} as well as in the analysis of the large time asymptotics of the solutions \cite{desvillettes2006exponential, das2024convergence}. 

In this article we analyze the effect of homogenization in the three species reaction-diffusion system. This model considers periodically distributed smooth holes inside the domain where mass flows from outside, enters the system and influences the reversible chemical kinetics. We also introduce a security zone near the boundary of the domain where no holes are present. A detailed description of the domain with periodically distributed holes (also called perforated domain) is provided in Section~\ref{description of the domain}. Here the evolution (in time) of densities of chemical species $X_1, X_2$ and $X_3$ are captured by the following model: 
\begin{equation}\label{eq: aeps}
    \left\{
    \begin{aligned}
        & \partial_t a^1_{\varepsilon} - d_1 \Delta a^1_{\varepsilon} = a^3_{\varepsilon} - a^1_{\varepsilon}a^2_{\varepsilon} \qquad && \mbox{in} \ (0,T) \times \Omega_{r_\varepsilon} \\
        & \nabla a^1_{\varepsilon} \cdot n_{\varepsilon} = \varepsilon \psi_1\left(t,x, \frac{x}{\varepsilon} \right) \qquad && \mbox{on} \ (0,T) \times \Gamma_{r_\varepsilon} \\
        & \nabla a^1_{\varepsilon} \cdot n = 0 \qquad && \mbox{on} \ (0,T) \times \partial \Omega \\
        & a^1_{\varepsilon}(0,x) = a^{1}_{\eps,0}(x) \qquad && \mbox{in}\  \Omega_{r_\eps}
    \end{aligned}
    \right.
\end{equation}
\begin{equation}\label{eq: beps}
    \left\{
    \begin{aligned}
        & \partial_t a^2_{\varepsilon} - d_2 \Delta a^2_{\varepsilon} = a^3_{\varepsilon} - a^1_{\varepsilon}a^2_{\varepsilon} \qquad && \mbox{in}\ (0,T) \times \Omega_{r_\varepsilon} \\
        & \nabla a^2_{\varepsilon} \cdot n_{\varepsilon} = \varepsilon \psi_2\left(t,x, \frac{x}{\varepsilon} \right) \qquad && \mbox{on} \ (0,T) \times \Gamma_{r_\varepsilon} \\
        & \nabla a^2_{\varepsilon} \cdot n = 0 \qquad && \mbox{on} \ (0,T) \times \partial \Omega \\
        & a^2_{\varepsilon}(0,x) = a^2_{\eps,0}(x) \qquad && \text{in}\  \Omega_{r_\eps}
    \end{aligned}
    \right.
\end{equation}
and
\begin{equation}\label{eq: ceps}
    \left\{
    \begin{aligned}
        & \partial_t a^3_{\varepsilon} - d_3 \Delta a^3_{\varepsilon} = a^1_{\varepsilon}a^2_{\varepsilon} - a^3_{\varepsilon} \qquad && \text{in}\  (0,T)\times \Omega_{r_\varepsilon} \\
        & \nabla a^3_{\varepsilon} \cdot n_{\varepsilon} = \varepsilon \psi_3\left(t,x, \frac{x}{\varepsilon} \right) \qquad && \text{on} \ (0,T) \times \Gamma_{r_\varepsilon} \\
        & \nabla a^3_{\varepsilon} \cdot n = 0 \qquad && \mbox{on}\ (0,T) \times \partial \Omega \\
        & a^3_{\varepsilon}(0,x) = a^3_{\eps,0}(x) \qquad && \text{in}\  \Omega_{r_\eps}.
    \end{aligned}
    \right.
\end{equation}
Here $\Omega_{r_\eps}\subset \mathbb{R}^3$ denotes a bounded perforated domain, i.e. the domain $\Omega\subset \mathbb{R}^3$ with periodically distributed smooth holes of diameter $r_\eps= \mathcal{O}(\eps^{\alpha})$ with a Security Zone (see Section \ref{description of the domain}). $\Gamma_{\eps}$ be the smooth boundary of the holes and $n_{\eps}(x)$ be the normal at a spatial point $x\in \Gamma_{\eps}$. The Neumann boundary data $\psi_i\left(t,x,\frac x\eps\right)$ along the boundary of the holes, is taken to be nonnegative for each $i=1,2,3$, which mathematically captures the inflow of chemical masses in the reversible chemical reaction through the holes. We assume $\psi \in C^1([0,T], B)$, where $B := C^1\big(\overline{\Omega}; C^1_{\#}(Y)\big)$ ($C^1_{\#}(Y)$ denotes the space of $Y$-periodic $C^1$ functions). For $i=1,2,3$, $a^i_{\eps,0}$ be the initial data (for details see Appendix \ref{sec: appendix well prepared initial condition}) which is sufficiently regular ($\mathrm{W}^{2,p}(\Omega_{r_{\eps}})$ for $p$ large enough), nonnegative and converges to the smooth nonnegative function $a^i_0$ in a suitable sense. We further assume that the initial datum satisfies the following compatibility condition
\begin{align}\label{compatibility condition}
    \nabla a^i_{\eps,0} \cdot n = \psi_{i}\left(0,x,\frac x\eps\right)=0, \ \forall \, i=1,2,3, \ \forall \, x\in \Gamma_{r_\eps}\cup\partial\Omega,
\end{align}
along with the following uniform bound
\begin{align}\label{CI}
\| a^i_{\eps,0}\|_{\mathrm{L}^2(\Omega_{r_\eps})}, \| a^i_{\eps,0}\|_{\mathrm{L}^4(\Omega_{r_\eps})}, \ \| \nabla a^i_{\eps,0}\|_{\mathrm{L}^2(\Omega_{r_\eps})} \leq C^I, \ \forall\, i=1,2,3,
\end{align}
where $C^I$ is some positive constant, uniform in $\eps$. 

Observe that along the boundary of the holes, we have non-homogeneous Neumann data. Homogenization problems with non-homogeneous Neumann data, which are also closely related to the homogenization problems in porous media have attracted significant attraction in the mathematical community. We would like to refer the readers to \cite{donato_neumann_nonhomogeneous, conca2003effective, franchi2016microscopic,hornung1991diffusion, krehel2014homogenization} for various works on this topic. 

Moreover, the source terms appearing on the right hand side of the system \eqref{eq: aeps}-\eqref{eq: ceps} are quadratic in nature, i.e., they are quadratic polynomials. Hence, our problem is also related to nonlinear quadratic parabolic homogenization problem. Similar systems are studied in \cite{franchi2016microscopic, desvillettes2018homogenization}, where authors consider a fragmentation-coagulation model (a particular reaction-diffusion model with infinitely many unknowns), arising from biomedical research such as in the study of Alzheimer disease \cite{franchi2016microscopic}. In \cite{desvillettes2018homogenization} due to the dominance of coagulation the first source term remains bounded from above by zero which need not hold for the three species reaction-diffusion system given by \eqref{eq: aeps}-\eqref{eq: ceps}.
Motivated by the work in \cite{desvillettes2018homogenization, franchi2016microscopic}, we study the effect of homogenization of the perforated domain to the system \eqref{eq: aeps}-\eqref{eq: ceps}. While the authors consider a quadratic model in \cite{desvillettes2018homogenization}, due to the dominance of coagulation the first source term always remains bounded from above by zero. This property need not hold for the three species reaction-diffusion system, given by \eqref{eq: aeps}-\eqref{eq: ceps}. Consequently, our approach for deriving various estimates differs significantly from theirs.

The homogenization problem in this case is to study the effective concentration of the species as the size of the perforations vanish. To study such an effect when the size of the perforations are of order of size $\eps$, we employ the method of two-scale convergence which has been widely used to study the effect of perforated domains for different models (see \cite{allaire_2scale, desvillettes2018homogenization, franchi2016microscopic, salvarani_2scale, nand_dirichlet, nand_neumann}). To apply two-scale convergence we need various uniform (in $\eps$) estimates(see Section \ref{sec: Uniform estimates}), after which passing to the limit as $\eps \to 0$ gives us the effective concentrations. The effective concentrations satisfies similar equation but with an extra reaction rate. 
We will now introduce our main result.

\begin{theorem}\label{main convergence theorem}
    Let the perforations $r_{\eps}$ is of the size $\mathcal{O}(\eps)$. Let $\aeps, \beps$ and $\ceps$ satisfies the system \eqref{eq: aeps}-\eqref{eq: ceps} respectively and the initial condition satisfies compatibility condition \eqref{compatibility condition}, uniform estimate \eqref{CI} and  
    \[
    a^i_{\eps,0}\xrightharpoonup[]{2-scale} a^{i}_0, \ \forall \, i=1,2,3.
    \]
    Then we have the following: for $i=1,2,3$ 
    \begin{equation}
        \left\{
        \begin{aligned}
            \Tilde{a^i_{\eps}} &\xrightharpoonup{2-scale} a^i
            \\
            \Tilde{ \nabla a^i_{\eps}} & \xrightharpoonup{2-scale} \nabla_x a^i(t,x) + \nabla_y a^{i,1}(t,x,y)
            \\
            \Tilde{\dert a^i_{\eps}} & \xrightharpoonup{2-scale} \dert a^i,
        \end{aligned}
        \right. 
    \end{equation}
    where $\Tilde{.}$ denotes the extension by zero outside $\Omega_{r_{\eps}}$
    and $a^i$ is the solution of the following equation: for $i = 1,2$ 
    \begin{equation}\label{eq: hom aeps}
    \left\{
    \begin{aligned}
        \theta \frac{\partial a^i}{\partial t}(t,x) 
        - d_i \dv \left[ D\nabla a^i(t,x) \right] 
        = & \theta \big( a^3(t,x) - a^1(t,x) a^2(t,x)\big) \\
        &+ d_i \int\limits_{\Gamma} \psi_i(t,x,y) 
        \qquad  \qquad  \text{in}\ (0,T) \times \Omega
        \\
         \theta \frac{\partial a^3}{\partial t}(t,x) 
        - d_3 \dv \left[ D\nabla a^3(t,x) \right] 
        =& \theta \big( a^1(t,x) a^2(t,x) - a^3(t,x) \big)\\
        &+ d_3 \int\limits_{\Gamma} \psi_3(t,x,y) 
        \qquad \qquad  \text{in} \ (0,T) \times \Omega
        \\
         D\nabla a^i(t,x) \cdot n =& 0 \qquad  \qquad \qquad \qquad \qquad \quad \text{on} \ (0,T) \times \partial \Omega 
        \\
        D\nabla a^3(t,x) \cdot n =& 0 \qquad  \qquad \qquad \qquad \qquad \quad \text{on} \ (0,T) \times \partial \Omega
        \\
        a^i(0,x) = & a^i_0(x) \qquad \qquad \qquad \qquad \qquad \text{in}\  \Omega 
        \\
         a^3(0,x) = &a^3_0(x) \qquad \qquad \qquad \qquad \qquad \text{in}\ \Omega. 
    \end{aligned}
    \right.
\end{equation}
Here $\theta = |Y^*|$, where $Y^*$ is as defined in \eqref{cond: boundary Gammaeps} and $D$ is a $3 \times 3$ matrix whose entries are given by 
\[
D_{ij}:= \int_{Y^*} \delta_{ij} + \frac{\partial \omega_j}{\partial y_i}
\]
where $\omega_j$ satisfies the following cell problem on the perforated unit cell $Y^*$
    \begin{equation}
    \left\{
    \begin{aligned}
        & -\dv_y [ \nabla_y\omega_j + e_j] = 0 \qquad Y^*
        \\
        & [ \nabla_y\omega_j + e_j]\cdot n = 0 \qquad \Gamma
        \\
        & \omega_j \in \mathrm H^1(Y^*) \text{ and is Y-periodic},
    \end{aligned}
    \right.
\end{equation} 
 and 
\[
a^{i,1}(t,x,y) = \omega_j(y) \frac{\partial a^i}{\partial x_j}(t,x), \ \forall\, i=1,2,3.
\] 
\end{theorem}
\begin{remark}
    Suppose $a^i_{\eps,0}$ is positive constant, where the constant is independent of $\eps$, for $i=1,2,3$, then a similar result is obtained with the same constant serving as the initial datum. In this case, constant initial data automatically satisfy the compatibility condition \eqref{compatibility condition} as well as the uniform estimate \eqref{CI}.
\end{remark}

The second result of our article concerns the case when the holes vanish rapidly, i.e., when their size is of order $\mathcal{O}(\varepsilon^{\alpha})$ for $\alpha > 1$.
\begin{theorem}\label{theorem: rapid main}
Assume that the perforations $r_{\eps}$ is of the size $\mathcal{O}(\eps^{\alpha})$ for $\alpha>1$. Let $\aeps, \beps$ and $\ceps$ satisfies the system \eqref{eq: aeps}-\eqref{eq: ceps} respectively and the initial condition satisfies compatibility condition \eqref{compatibility condition}, uniform estimate \eqref{CI} and  
    \[
    \Tilde{a^i_{\eps,0}}\to a^{i}_0, \ \mbox{strongly in } \mathrm L^2((0,T);\Omega)\ \forall \, i=1,2,3.
    \]
    Then we have the following: for $i=1,2,3$ 
    \begin{equation}
        \Tilde{a^i_{\eps}} \to \mathfrak{v}_i, \ \text{strongly in } \mathrm L^2((0,T);\Omega)
    \end{equation}
    where $\mathfrak{v}_i$ solves the following equation: for $i = 1,2$
    \begin{equation*}
    \left\{
    \begin{aligned}
    \dert \mathfrak{v}_i 
        - d_i \Delta \mathfrak{v}_i 
        & = \mathfrak{v}_3 - \mathfrak{v}_1\mathfrak{v}_2 
        \qquad && \text{in } (0,T) \times \Omega
        \\
        \dert \mathfrak{v}_3 
        - d_3 \Delta \mathfrak{v}_3 
        & =  \mathfrak{v}_1\mathfrak{v}_2 - \mathfrak{v}_3
        \qquad && \text{in } (0,T) \times \Omega
        \\
        \nabla \mathfrak{v}_i \cdot n & = 0  \qquad && \text{on } (0,T) \times \partial \Omega
        \\
        \nabla \mathfrak{v}_3 \cdot n & = 0  \qquad && \text{on } (0,T) \times \partial \Omega
        \\
        \mathfrak{v}_i(0,\cdot) & = a_0^i(x) \qquad && \text{in } (0,T) \times \partial \Omega
        \\
        \mathfrak{v}_3(0,\cdot) & = a_0^3(x) \qquad && \text{in } (0,T) \times \partial \Omega
    \end{aligned}
    \right.
    \end{equation*}
\end{theorem}

\begin{remark}
    A more general study could be possible with the non-homogeneous Neumann data is of the form $\eps^{\beta}\psi\left(t,x,\frac{x}{\eps}\right)$, for some $\beta\geq 0$. A thorough investigation of this question is left for future work.
\end{remark}

\subsection{Description of Domain}\label{description of the domain}

\begin{figure}[h] 
    \centering
    \begin{minipage}{0.45\textwidth}
        \centering
        \begin{tikzpicture}[remember picture, scale=3]
   
    \def\outerR{1.0}
    \def\innerR{0.8}
    \def\holeR{0.02}      
    \def\side{0.1}       
    \def\halfside{0.05}   

    \draw[thick, fill=green!30] (0,0) circle (\outerR);

    \begin{scope}
        \clip (0,0) circle (\innerR);

        \foreach \x in {-0.9, -0.8, ..., 0.9} {
            \foreach \y in {-0.9, -0.8, ..., 0.9} {
                
                \pgfmathparse{sqrt(\x*\x + \y*\y)}
                \let\dist\pgfmathresult
                
%
                
                \pgfmathparse{(\dist + \holeR < \innerR) ? 1 : 0}
                \ifnum\pgfmathresult=1
                    \draw[thin, gray, dotted] (\x-\halfside, \y-\halfside) rectangle (\x+\halfside, \y+\halfside);
                    
                    \draw[fill=white, thin] (\x,\y) circle (\holeR);
                    \draw[fill=white, thin] (0,0) circle (\holeR);
    		   \coordinate (holecenter) at (0,0);
                \fi

            }
        }
    \end{scope}

    \draw[thick, dotted] (0,0) circle (\innerR);

    \draw[black, thick] (\innerR, 0) -- (\outerR, 0) 
        node[midway, above, text=black] {\small $\delta$};

\end{tikzpicture}
        \caption{Domain with periodically distributed holes.}
        \label{fig: domain}
    \end{minipage}
    \hfill
    \begin{minipage}{0.45\textwidth}
        \centering
        \begin{tikzpicture}[remember picture, scale=3]
  
   \draw[thick, fill=green!30] (0, 0) rectangle (1,1);
   
\node at (0.15, 0.85) {$Y^*$};
     \draw[fill=white, thin] (0.5,0.5) circle (0.25);
 	\node at (0.5, 0.5) {$T$};
	\coordinate (unitcenter) at (0.4, 0.5);

\end{tikzpicture} 
        \caption{Perforated unit cell}
        \label{fig: unit cell}
    \end{minipage}
    \begin{tikzpicture}[remember picture, overlay]
        \draw[-, >=stealth, bend left=20, very thick, dotted, red] (holecenter) to node[midway, above, text=blue] {\small } (unitcenter);
    \end{tikzpicture}
\end{figure}

In this paper we reserve the standard notation $Y$ for the unit periodic cell $\left[-\frac{1}{2},\frac{1}{2} \right)^3$. Let $T \subset int(Y)$ be such that it has a smooth boundary $\Gamma$. Define $Y^*$ as 
\begin{align}\label{area outside the hole in unit cube}
Y^* = Y - \overline{T}
\end{align}
which is sometimes called as solid part or material part (see Figure \ref{fig: unit cell}).
Let $\Omega$ be an open, bounded and connected subset of $\RR^3$ with smooth boundary $\partial \Omega$. Fix $\delta>0$ and define the following set 
\begin{equation}\label{security zone}
    \Omega^{\delta} := \{ x \in \Omega: \dist{(x,\partial \Omega)} > \delta \} 
\end{equation}
and $\Omega \setminus \Omega^{\delta}$ is referred to as "Security zone"(see \cite{MR1676922, MR1932963}).
For parameter $\eps>0$ define the indexing sets $\mathcal I_{\eps}$ and $\mathcal J_{\eps}$ as 
\begin{equation}\label{index of holes}
    \begin{aligned}
        \mathcal I_{\eps}& := \{ k \in \ZZ^3: (\eps k + r_\eps\overline{T}) \subset \Omega^{\delta} \}
        \\
        \mathcal J_{\eps}& := \{ k \in \ZZ^3:  \Omega \cap (\eps k + \eps Y) \text{ is non-empty} \} \setminus \mathcal I_{\eps}
    \end{aligned}
\end{equation}
Let us define the following sequence of real numbers $\sigma_{\eps}$ as follows 
\[
\sigma_{\eps}:= \frac{r_{\eps}}{\eps}
\]
Let us define the sets $R_{\eps}$ and $S_{\eps}$ as follows
\begin{equation*}
    \begin{aligned}
        R_{\eps}&:= \bigcup_{k\in \mathcal I_{\eps}} \big( \eps k + \eps (Y\setminus \sigma_{\eps}\overline{T} ) \big) 
        \\
        S_{\eps}& := \bigcup_{k \in \mathcal{J}_{\eps}}(\eps k + \eps Y).
    \end{aligned}
\end{equation*}
Hence, we define our periodically perforated domain $\Omega_{r_{\eps}} $(see \cite{desvillettes2018homogenization, franchi2016microscopic})
\[
\Omega_{r_{\eps}}:=  \Omega \cap (R_{\eps}\sqcup S_\eps) .
\]
Observe that $\partial \Omega_{r_{\eps}}$ comprises of two boundaries 
\[
\partial \Omega_{r_{\eps}} = \partial \Omega \cup \Gamma_{\eps},
\]
where $\Gamma_{\eps}$ is the union of the boundary of all the holes lying inside $\Omega^{\delta}$ (see Figure \ref{fig: domain}), i.e.,
\[
\Gamma_{\eps}:= \bigcup_{k\in \mathcal I_{\eps}} \left\{ x \in \Omega : x \in \eps k + r_\eps \Gamma \right\}.
\]
Let 
\[
\gamma  := \lim_{\eps \to 0}\eps \left| \Gamma_{\eps} \right| .
\]
Observe that 
\begin{equation}\label{cond: boundary Gammaeps}
\left\{
\begin{aligned}
    & \gamma = | \Gamma | \frac{|\Omega|}{|Y|}, \qquad && r_{\eps} = \mathcal{O}(\eps)
    \\
    & \gamma = 0, \qquad && r_{\eps} = \mathcal{O}(\eps^{\alpha}) \text{ for } \alpha >1
    \\
    & \eps^3 \times \# \mathcal{I}_{\eps} \leq \gamma_1,
    \end{aligned}
    \right.
\end{equation}
where $\# \mathcal{I}_{\eps}$ denote the number of holes in $\Omega_{r_{\eps}}$ and $\gamma_1$ is a positive constant independent of $\eps$(see \cite{AllaireDamlamianHornung1996}).

\subsection{Plan of the Paper}
We divide this article in two sections and several appendices. In Section \ref{sec: Uniform estimates} we derive the uniform estimates which is required for the two-scale convergence in Lemma \ref{L2 integral estimate}- Proposition \ref{L2 integral estimate time derivative}. In Section \ref{sec: Homogenization eps}, we give the prove the Theorem \ref{main convergence theorem} using the two-scale convergence method. In Section \ref{sec: Homogenization rapid}, we derive the Theorem \ref{theorem: rapid main} by passing the limit using the compactness in $\mathrm L^2$ for rapidly decaying hole. In Appendix \ref{sec: appendix existence theory} we derive the global-in-time existence of the solution to the system \eqref{eq: aeps}-\eqref{eq: ceps}, and obtain regularity estimates. In Appendix \ref{sec: appendix extension operator}-\ref{sec: appendix two-scale}, we discuss various classical tools which are extensively used in our article such as the existence of extension operator in perforated domain, the trace inequality, Anisotropic Sobolev inequality and two scale convergence method. In the final Appendix \ref{sec: appendix well prepared initial condition}, we construct well-prepared initial data demonstrating that the homogenization of the model considered in this article can also be carried out for certain non-constant initial data. These initial datum are nonnegative, sufficiently smooth, satisfy the compatibility condition \eqref{compatibility condition}, uniform estimates as described in \eqref{CI}, and converge to smooth, nonnegative functions in the two-scale sense.    

\section*{Acknowledgment}
SD acknowledges funding by the Department of Atomic Energy (DAE), Government of India, in the form of Postdoctoral Research Fellowship. KS acknowledges funding by the NBHM-DAE (Government of India), in the form of PhD Fellowship.

\section{Uniform Estimates}\label{sec: Uniform estimates}

We start with the following uniform (in $\eps$) $\mathrm{L}^2$-integral estimate of the solution. It can be proved along a similar line as in \cite[Lemma 2.1]{desvillettes2018homogenization}, \cite{franchi2016microscopic}. The same estimate when there is no perforation in the domain can be found in \cite{desvillettes2007global}. Thus we omit the proof of the result.
\begin{lemma}\label{L2 integral estimate}
    Let $a^1_{\eps},a^2_{\eps}$ and $a^3_{\eps}$ be the solution to the system \eqref{eq: aeps}-\eqref{eq: ceps}. Then we have 
    \begin{equation}\label{uniform: L2L2}
        \Vert a^i_{\eps} \Vert_{\mathrm L^2((0,T);\mathrm L^2(\Omega_{r_{\eps}}))} \leq \kappa_1, \ \forall \, 1 \leq i \leq 3,
    \end{equation}
    where $\kappa_1$ is a positive constant, independent of $\eps$.
\end{lemma}

We now move to establish similar uniform (in $\eps$) $\mathrm{L}^2$-integral estimate of the gradient of the solution. We have the following lemma.
  
\begin{lemma}\label{anisotropic norm estimate}
    Let $a^1_{\eps},a^2_{\eps}$ and $a^3_{\eps}$ be the solution to the system \eqref{eq: aeps}-\eqref{eq: ceps}. Then we have 
    \begin{equation}\label{uniform: LinftyL2 and grad L2L2 }
        \Vert a^i_{\eps} \Vert_{\mathrm L^{\infty}((0,T);\mathrm L^2(\Omega_{r_{\eps}}))} + \Vert \nabla a^i_{\eps} \Vert_{\mathrm L^2((0,T);\mathrm L^2(\Omega_{r_{\eps}}))} \leq \kappa_2 \qquad \forall 1 \leq i \leq 3,
    \end{equation}
    where $\kappa_2$ is a positive constant, independent of $\eps$.
\end{lemma}
\begin{proof}
    Equation satisfied by $a^1_{\eps}$ can be rewritten as follows
    \[
        \partial_t a^1_{\varepsilon} - d_2 \Delta a^1_{\varepsilon} + a^1_{\varepsilon}a^2_{\varepsilon} = a^3_{\varepsilon}. 
    \]
    Since $a^i_{\eps}$ is nonnegative (see Theorem \ref{existence, uniqueness,regularity, epsilon scale}) for all $1\leq i\leq 3$. Testing with $a^1_{\eps}$ yields
    \[
    \int\limits_{\Omega_{r_{\eps}}} a^1_{\eps} \partial_t a^1_{\eps} - d_1 \int\limits_{\Omega_{r_{\eps}}} a^1_{\eps} \Delta a^1_{\eps} \leq \int\limits_{\Omega_{r_{\eps}}} a^1_{\eps} a^3_{\eps}. 
    \]
    Integrating the above inequality with respect to the time variable from $0\leq s \leq T$ we obtain the following 
    \begin{align*}
    \int\limits_0^s\int\limits_{\Omega_{r_{\eps}}} a^1_{\eps} \partial_t a^1_{\eps} + d_1 \int\limits_0^s\int\limits_{\Omega_{r_{\eps}}} |\nabla a^1_{\eps}|^2 - d_1 \int\limits_0^s\int\limits_{\Gamma_{\eps}} a^1_{\eps} \frac{\partial a^1_{\eps} }{\partial n_{\eps}} 
    & \leq \int\limits_0^s\int\limits_{\Omega_{r_{\eps}}} a^1_{\eps} a^3_{\eps}. 
    \end{align*}
    Using the boundary condition from \eqref{eq: aeps} we can write it as 
    \begin{align*}
    \int\limits_0^s\int\limits_{\Omega_{r_{\eps}}} \partial_t \big( a^1_{\eps} \big)^2
    + 2d_1 \int\limits_0^s\int\limits_{\Omega_{r_{\eps}}} |\nabla a^1_{\eps}|^2 
    - &2 d_1 \eps \int\limits_0^s\int\limits_{\Gamma_{\eps}} a^1_{\eps} \psi_1\left(t,x,\frac{x}{\eps} \right)
     \leq 2\int\limits_0^s\int\limits_{\Omega_{r_{\eps}}} a^1_{\eps} a^3_{\eps} 
    \\
    & \leq 2\Vert a^1_{\eps} \Vert_{\mathrm L^2((0,T);\mathrm L^2(\Omega_{r_{\eps}}))} \Vert a^3_{\eps} \Vert_{\mathrm L^2((0,T);\mathrm L^2(\Omega_{r_{\eps}}))} 
    \\
    & \leq 2\kappa_1. 
    \end{align*}
    where we have used H\"older for the second inequality and \eqref{uniform: L2L2} for the third inequality. Integration with respect to time variable on the first integral yields 
    \begin{equation}\label{ineq: grad proof}
    \int\limits_{\Omega_{r_{\eps}}} \big( a^1_{\eps}(s,\cdot) \big)^2
    + 2d_1 \int\limits_0^s\int\limits_{\Omega_{r_{\eps}}} |\nabla a^1_{\eps}|^2 
    - 2 d_1 \eps \int\limits_0^s\int\limits_{\Gamma_{\eps}} a^1_{\eps} \psi_1\left(t,x,\frac{x}{\eps} \right)
    \leq 2\kappa_1 
    + \int\limits_{\Omega_{r_{\eps}}} \big( a^1_{\eps,0} \big)^2 =: C,
    \end{equation}
    where we have used the niform bound of the initial condition \eqref{CI}. Consider the boundary integral
    \begin{align*}
        2d_1 \varepsilon \int_0^{s} \int_{\Gamma_\varepsilon} \psi_1(t, x, \tfrac{x}{\varepsilon}) a_1^\varepsilon \, ds(x) dt 
        &\le 2d_1 \varepsilon \|\psi_1\|_{B} \int_0^{s} \int_{\Gamma_\varepsilon} |a_1^\varepsilon| \, ds(x) dt \\
        &\leq 2d_1 \varepsilon \|\psi_1\|_{B} \sqrt{s|\Gamma_\varepsilon|} \|a_1^\varepsilon\|_{L^2((0, s); \mathrm L^2(\Gamma_\varepsilon))}.
    \end{align*}
Employing Lemma \ref{app: trace theory}, we obtain
     \begin{equation}\label{ineq: grad prooof bdry}
     \begin{aligned}
        2d_1 \varepsilon \int_0^{s} \int_{\Gamma_\varepsilon} \psi_1(t, x, \tfrac{x}{\varepsilon}) a_1^\varepsilon \, ds(x) dt   &\le 2C_{tr}d_1 \sqrt{T} \|\psi_1\|_{B} \sqrt{\varepsilon |\Gamma_\varepsilon|} \|a_1^\varepsilon\|_{L^2((0, s); \mathrm{ H}^1(\Omega_{r_{\eps}}))}  
        \\
        &\le 2\gamma C_{tr}d_1 \sqrt{T} \cdot \|\psi_1\|_{B} \|a_1^\varepsilon\|_{L^2((0, s); \mathrm{ H}^1(\Omega_{r_{\eps}}))} 
        \\
        & = \widetilde{C} \left[ \|a_1^\varepsilon\|_{L^2((0, s); L^2(\Omega_{r_{\eps}}))} + \|\nabla a_1^\varepsilon\|_{L^2((0, s); L^2(\Omega_{r_{\eps}}))} \right],
        \end{aligned}
        \end{equation}
where $\widetilde{C}:=2\gamma C_{tr}d_1 \sqrt{T} \cdot \|\psi_1\|_{B}$ and $\gamma$ as defined in \eqref{cond: boundary Gammaeps}. From \eqref{ineq: grad proof} and \eqref{ineq: grad prooof bdry} we obtain the following 
\begin{equation*}
    \begin{aligned}
    \int\limits_{\Omega_{r_{\eps}}} \big( a^1_{\eps}(s,\cdot) \big)^2
    + 2d_1 \int\limits_0^s\int\limits_{\Omega_{r_{\eps}}} |\nabla a^1_{\eps}|^2 
    & \leq C 
    + \Tilde{C} \|a_1^\varepsilon\|_{L^2((0, s); L^2(\Omega_{r_{\eps}}))} 
    + \Tilde{C} \|\nabla a_1^\varepsilon\|_{L^2((0, s); L^2(\Omega_{r_{\eps}}))}
    \\
    & \leq C 
    + \Tilde{C}\|a_1^\varepsilon\|_{L^2((0, T); L^2(\Omega_{r_{\eps}}))} 
    + \frac{\Tilde{C}^2}{4d_1}
    + d_1 \| \nabla a_1^\varepsilon\|_{L^2((0, s); L^2(\Omega_{r_{\eps}}))}^2
    \\
    & \leq C + \Tilde{C} \kappa_1+ \frac{\Tilde{C}^2}{4d_1}+ d_1 \| \nabla a_1^\varepsilon\|_{L^2((0, s); L^2(\Omega_{r_{\eps}}))}^2,
    \end{aligned}
\end{equation*}
where we have used Young's inequality and \eqref{uniform: L2L2} as described in Lemma \ref{L2 integral estimate}. Choice of the constant as $C_1:=C + \Tilde{C} \kappa_1+ \frac{\Tilde{C}^2}{4d_1}$, yields
\begin{equation}
    \int\limits_{\Omega_{r_{\eps}}} \big( a^1_{\eps}(s,\cdot) \big)^2
    + d_1 \int\limits_0^s\int\limits_{\Omega_{r_{\eps}}} |\nabla a^1_{\eps}|^2 
    \leq C_1.
\end{equation}
Since above inequality holds true for $0\leq s \leq T$ we obtain the following 
\begin{equation}
  \| a^1_{\eps} \|^2_{\mathrm L^{\infty}((0,T);\mathrm L^2(\Omega_{r_{\eps}}))} 
  + d_1 \| \nabla a^1_{\eps} \|^2_{\mathrm L^{2}((0,T);\mathrm L^2(\Omega_{r_{\eps}}))} 
  \leq C_1.
\end{equation}
Similar calculations can be performed for the $a^2_{\eps}$. Next we move towards the estimate corresponding to $a^3_{\eps}$. Testing equation \eqref{eq: ceps} with $a^3_{\eps}$ gives us
\[
        \int\limits_{0}^s \int\limits_{\Omega_{r_{\eps}}} a^3_{\eps} \partial_t a^3_{\eps} 
        - d_3 \int\limits_{0}^s \int\limits_{\Omega_{r_{\eps}}}a^3_{\eps} \Delta a^3_{\eps} 
        + \int\limits_{0}^s \int\limits_{\Omega_{r_{\eps}}} \big(a^3_{\eps} \big)^2
        = \int\limits_{0}^s \int\limits_{\Omega_{r_{\eps}}} a^1_{\eps}a^2_{\eps}a^3_{\eps}.
 \]
for some $0\leq s \leq T$. Integration by parts yields
\begin{align}
        \int\limits_{\Omega_{r_{\eps}}} \left| a^3_{\eps}(s,\cdot) \right|^2
        -& \int\limits_{\Omega_{r_{\eps}}} \left| a^3_{\eps}(0,\cdot) \right|^2
        + 2d_3 \int\limits_{0}^s \int\limits_{\Omega_{r_{\eps}}} \left| \nabla a^3_{\eps} \right|^2 \nonumber \\
         & \leq 2d_3 \eps \int\limits_{0}^s \int\limits_{\Omega_{r_{\eps}}} \psi_3\left(t,x,\frac{x}{\eps} \right) a^3_{\eps}
        + \int\limits_{0}^s \int\limits_{\Omega_{r_{\eps}}} a^1_{\eps}a^2_{\eps}a^3_{\eps} =:J_1 + J_2. \label{ineq: grad3 uniform}
\end{align}
The first term $J_1$ can be handled in a similar fashion as done in \eqref{ineq: grad prooof bdry} and hence there exists positive constant $C_2$, independent of $\eps$, such that 
\begin{equation}\label{ineq: J1}
J_1 \leq C_2 + C_2\| \nabla a^3_{\eps} \|_{\mathrm L^{2}((0,s);\mathrm L^2(\Omega_{r_{\eps}}))}   .  
\end{equation}
Let us now look at $J_2$
\begin{equation}\label{exp: J_2}
\begin{aligned}
J_2 & =\int_0^{s} \int_{\Omega_{r_{\eps}}} a^1_{\eps} a^2_{\eps} a_3^\varepsilon \le \int_0^{s} \left[ \int_{\Omega_{r_{\eps}}} \big(a^1_{\eps} a^2_{\eps}\big)^2 \right]^{1/2} \left[ \int_{\Omega_{r_{\eps}}} \big(a^3_{\eps}\big)^2 \right]^{1/2} 
\\
&\le \sup_{s \in [0, T]} \left[ \int_{\Omega_{r_{\eps}}} \big(a^3_{\eps}(s,\cdot)\big)^2 \right]^{1/2}  \int_0^T \left[ \int_{\Omega_{r_{\eps}}} \big(a^1_{\eps}\big)^2 \big(a^2_{\eps}\big)^2 \right]^{1/2} 
\\
&\le \frac{1}{\sqrt{2}} \sup_{s \in [0, T]} \left( \int_{\Omega_{r_{\eps}}} \big(a^3_{\eps}(s,\cdot)\big)^2 \right)^{1/2} \int_0^T \left[ \int_{\Omega_{r_{\eps}}} \big(a^1_{\eps}\big)^4 + \int_{\Omega_{r_{\eps}}} (a^2_{\eps})^4 \right]^{1/2}
\\
&\le \frac{1}{\sqrt{2}}  \left( \sup_{s \in [0, T]} \int_{\Omega_{r_{\eps}}} \big(a^3_{\eps}(s,\cdot)\big)^2 \right)^{1/2}  \left[ \int_0^T \left( \int_{\Omega_{r_{\eps}}} (a^1_{\eps})^4 \right)^{1/2} + \int_0^T \left( \int_{\Omega_{r_{\eps}}} (a^2_{\eps})^4 \right)^{1/2} \right].
\end{aligned}
\end{equation}
Invoking Anisotropic Sobolev inequality as described in Lemma \ref{app: anisotropic sobolev ineq} for $r_1 = 2$ and $q_1 = 6$ we obtain 
\[
\left(\int\limits_{0}^T \| a^i_{\eps} \|^2_{\mathrm L^6(\Omega_{r_{\eps}})} \right)^{1/2} \leq C_{AS} \left[ \| a^i_{\eps} \|_{\mathrm L^{\infty}((0,T);\mathrm L^2(\Omega_{r_{\eps}}))} 
  + \| \nabla a^i_{\eps} \|_{\mathrm L^{2}((0,T);\mathrm L^2(\Omega_{r_{\eps}}))}  \right] 
  \qquad \text{ for }i=1,2.
\]
Using H\"older and estimate \eqref{uniform: LinftyL2 and grad L2L2 } for $i=1,2$,  we obtain 
\[
\| a^i_{\eps} \|_{\mathrm L^{2}((0,T);\mathrm L^4(\Omega_{r_{\eps}}))} \leq C_3 \qquad \text{ for }i=1,2,
\]
where $C_3$ is some constant independent of $\eps$. Hence using this inequality along with the estimate \eqref{exp: J_2} we obtain the following 
\begin{equation}\label{ineq: J2}
    J_2 \leq \sqrt{2}C_3 \left( \sup_{0\leq s \leq T} \int_{\Omega_{r_{\eps}}} \big(a^3_{\eps}(s,\cdot)\big)^2 \right)^{1/2}.
\end{equation}
The estimates \eqref{ineq: grad3 uniform}, \eqref{ineq: J1}, \eqref{ineq: J2} yields the following 
\begin{equation*}
    \begin{aligned}
        \int\limits_{\Omega_{r_{\eps}}} \left| a^3_{\eps}(s,\cdot) \right|^2
        - \int\limits_{\Omega_{r_{\eps}}} &\left| a^3_{\eps}(0,\cdot) \right|^2
        + 2d_3 \int\limits_{0}^s \int\limits_{\Omega_{r_{\eps}}} \left| \nabla a^3_{\eps} \right|^2
        \\
        & \leq C_2 
        + \sqrt{2}C_3 \left( \sup_{0\leq s \leq T} \int_{\Omega_{r_{\eps}}} \big(a^3_{\eps}(s,\cdot)\big)^2 \right)^{1/2} 
        + C_2 \| \nabla a^3_{\eps} \|_{\mathrm L^{2}((0,s);\mathrm L^2(\Omega_{r_{\eps}}))}. 
    \end{aligned}
\end{equation*}
Using Young's inequality we can further reduce the above inequality as follows 
\begin{equation*}
    \begin{aligned}
        \int\limits_{\Omega_{r_{\eps}}} \left| a^3_{\eps}(s,\cdot) \right|^2
        + d_3 \int\limits_{0}^s \int\limits_{\Omega_{r_{\eps}}} \left| \nabla a^3_{\eps} \right|^2
        & \leq C_4 
        + \frac{1}{2}  \sup_{0\leq s \leq T} \int_{\Omega_{r_{\eps}}} \big(a^3_{\eps}(s,\cdot)\big)^2, 
    \end{aligned}
\end{equation*}
where $C_4>0$ some constant independent of $\eps$. Since above inequality is true for any $0 \leq s \leq T$, taking supremum over $s$ gives us the desired result.
\end{proof}

Next, we move onto the uniform (in $\eps$) $\mathrm{L}^2$-integral estimate of the time derivative of the solution. We start with the following Proposition. 
\begin{prop}\label{L4 norm estimate}
    Let $a^1_{\eps}, a^2_{\eps}$ and $a^3_{\eps}$ be the solution to the system \eqref{eq: aeps}-\eqref{eq: ceps}. Then we have 
    \begin{equation}\label{uniform: linfty l4}
        \Vert a^i_{\eps} \Vert_{\mathrm L^{\infty}\big((0,T);\mathrm L^4(\Omega_{r_{\eps}})\big)} \leq \kappa_3, \  \forall  i = 1, 2.
    \end{equation}
    where $\kappa_3$ is independent of $\eps$.
\end{prop}
\begin{proof}
    Multiplying equation \eqref{eq: aeps} by $2a^1_{\eps}$ yields
    \[
    2\aeps \dert \aeps - 2d_1 \aeps \Delta \aeps = 2 \aeps \ceps - 2 \big( \aeps \big)^2 \beps,
    \]
    which can further be rewritten as 
    \begin{align}\label{L4-1st}
    \dert \big( \aeps \big)^2 
    - d_1\big[\Delta \big( \aeps \big)^2 - 2|\nabla \aeps |^2 \big] 
    + 2 \big( \aeps \big)^2 \beps
    = 2 \aeps \ceps,
    \end{align}
    where we have used the following identity 
    \[
    \Delta u^2 = 2u\Delta u + 2|\nabla u|^2.
    \]
    Testing equation \eqref{L4-1st} with $\big( \aeps \big)^2$ on $(0,s) \times \Omega_{r_{\eps}}$ yields 
    \begin{align*}
    \int_0^s \int_{\Omega_{r_{\eps}}} \big(\aeps\big)^2 \partial_t \big(\aeps\big)^2 
    - & d_1 \int_0^s \int_{\Omega_{r_{\eps}}} \big(\aeps\big)^2 \Delta \big(\aeps\big)^2 
    \\
    +& 2d_1 \int_0^s \int_{\Omega_{r_{\eps}}} \big(\aeps\big)^2|\nabla \aeps|^2 
    + 2 \int_0^s \int_{\Omega_{r_{\eps}}} \big(\aeps\big)^4 \beps 
    = 2 \int_0^s \int_{\Omega_{r_{\eps}}} \big(\aeps\big)^3 \ceps.
    \end{align*}
    Since $a_{\eps}^2$ is nonnegative (see Theorem \ref{existence, uniqueness,regularity, epsilon scale}), the third and fourth integrals are non-negative. Hence we obtain the following
    \begin{align*}
     \int_0^s \int_{\Omega_{r_{\eps}}} \big( \aeps \big)^2 \partial_t \big( \aeps \big)^2 
     - d_1 \int_0^s \int_{\Omega_{r_{\eps}}} \big( \aeps \big)^2 \Delta \big( \aeps \big)^2 \le 2 \int_0^s \int_{\Omega_{r_{\eps}}} \ceps \big( \aeps \big)^2.
    \end{align*}
    Integration by parts yields 
    \begin{align*}
    \frac{1}{2} \int_0^s \int_{\Omega_{r_{\eps}}} \partial_t \big( \aeps \big)^4 + d_1 \int_0^s \int_{\Omega_{r_{\eps}}} \nabla (a_1^\eps)^2 \cdot \nabla (a_1^\eps)^2 - d_1 \int_0^s \int_{\partial \Omega_{r_{\eps}}} \big( \aeps \big)^2 \frac{\partial (a_1^\eps)^2}{\partial \eta} 
    \le 2 \int_0^s \int_{\Omega_{r_{\eps}}} \ceps \big( \aeps \big)^3.
    \end{align*}
    The above relation further simplifies to
    \begin{align*}
    \int_{\Omega_{r_{\eps}}} \big( \aeps \big)^4(s,\cdot) 
    + &2d_1 \int_0^s \int_{\Omega_{r_{\eps}}} \left|\nabla \big( \aeps \big)^2\right|^2 
    \\
    &\le \int_{\Omega_{r_{\eps}}} \big( a^1_{\eps,0} \big)^4 
    + 4 \int_0^s \int_{\Omega_{r_{\eps}}} \ceps \big( \aeps \big)^3
    + 2d_1 \int_0^s \int_{\partial\Omega_{r_{\eps}}} \big( \aeps \big)^2 \frac{\partial \big( \aeps \big)^2}{\partial \eta_\eps}.
\end{align*}
Recall the following identity 
\[
\frac{\partial \big( \aeps \big)^2}{\partial \eta_\eps} = 2 \aeps \frac{\partial \aeps}{\partial \eta_\eps}.
\]
Using the boundary condition from \eqref{eq: aeps} we obtain 
\begin{equation}\label{ineq: l4 deciding ineq}
    \begin{aligned}
    \int_{\Omega_{r_{\eps}}} \big( \aeps \big)^4(s,\cdot) 
    &+ 2d_1 \int_0^s \int_{\Omega_{r_{\eps}}} \left|\nabla \big( \aeps \big)^2\right|^2 \\
    & \le \int_{\Omega_{r_{\eps}}} \big( a^1_{\eps,0} \big)^4 
    + 4 \int_0^s \int_{\Omega_{r_{\eps}}} \ceps \big( \aeps \big)^3+ 4d_1 \eps \int_0^s \int_{\Gamma_\eps} \big( \aeps \big)^3 \psi_1\left(t,x,\frac{x}{\eps} \right).
\end{aligned}
\end{equation}
Dropping the second term in the above estimate yields: for any $0 \leq t \leq s \leq T$ 
\[
    \int_{\Omega_{r_{\eps}}} \big( \aeps \big)^4(t,\cdot) \leq \int_{\Omega_{r_{\eps}}} \big( a^1_{\eps,0} \big)^4 
    + 4 \int_0^s \int_{\Omega_{r_{\eps}}} \ceps \big( \aeps \big)^3
    + 4d_1 \eps \int_0^s \int_{\Gamma_\eps} \big( \aeps \big)^3 \psi_1\left(t,x,\frac{x}{\eps} \right).
\]
Integrating with respect to time $t$ from $0$ to $s$ we obtain 
\[
    \int_0^s\int_{\Omega_{r_{\eps}}} \big( \aeps \big)^4(t,\cdot) 
    \leq s \int_{\Omega_{r_{\eps}}} \big( a^1_{\eps,0} \big)^4 
    + 4s \int_0^s \int_{\Omega_{r_{\eps}}} \ceps \big( \aeps \big)^3
    + 4sd_1 \eps \int_0^s \int_{\Gamma_\eps} \big( \aeps \big)^3 \psi_1\left(t,x,\frac{x}{\eps} \right).
\]
Combining the above inequality with $\eqref{ineq: l4 deciding ineq}$ we obtain: for a.e. $ 0 \le t \le s \le T$
\begin{equation}\label{ineq: l4 updated deciding ineq}
\begin{aligned}
    \sup_{0 \le t \le s} \int_{\Omega_{r_{\eps}}} \big( \aeps \big)^4(t,\cdot) 
    +& \min\{1,2d_1\} \left \Vert \big( \aeps \big)^2 \right \Vert_{\mathrm L^2((0,T);\mathrm H^1(\Omega_{r_{\eps}}))}^2 
    \\
     = \sup_{0 \le t \le s} \int_{\Omega_{r_{\eps}}} \big( \aeps \big)^4(t,\cdot) 
    +& \min\{1,2d_1\} \left[ \int_0^s\int_{\Omega_{r_{\eps}}} \big( \aeps \big)^4 
    + \int_0^s \int_{\Omega_{r_{\eps}}} \left|\nabla \big( \aeps \big)^2\right|^2 \right] 
    \\
     \le (2+s)\int_{\Omega_{r_{\eps}}} \big( a^1_{\eps,0} \big)^4 
    + &4(2+s) \int_0^s \int_{\Omega_{r_{\eps}}} \ceps \big( \aeps \big)^3
  \\
    +& 4d_1(2+s) \eps \int_0^s \int_{\Gamma_\eps} \big( \aeps \big)^3 \psi_1\left(t,x,\frac{x}{\eps} \right). 
\end{aligned}
\end{equation}
We estimate the boundary integral in the following way:
    \begin{align*}
        \eps \int_0^s \int_{\Gamma_\eps} \big( \aeps \big)^3 \psi_1\left(t,x,\frac{x}{\eps} \right)
        & \leq \eps \Vert \psi_1 \Vert_{B} \int_0^s \int_{\Gamma_\eps} \left[\big( \aeps \big)^2 \right]^{3/2}
        \\
        & \leq \eps \Vert \psi_1 \Vert_{B} \big[T|\Gamma_{\eps}|\big]^{1/4} \left \Vert \big( \aeps \big)^2 \right \Vert_{\mathrm L^2((0,s);\mathrm L^2(\Gamma_{\eps}))}^{3/2}
        \\
        & \leq C_{tr}^{\frac 32}T^{1/4} \Vert \psi_1 \Vert_{B} \big[\eps|\Gamma_{\eps}|\big]^{1/4} \left \Vert \big( \aeps \big)^2 \right \Vert_{\mathrm L^2((0,s);\mathrm H^1(\Omega_{r_{\eps}}))}^{3/2}
        \\
        & \leq C_5 \left \Vert \big( \aeps \big)^2 \right \Vert_{\mathrm L^2((0,T);\mathrm H^1(\Omega_{r_{\eps}}))}^{3/2},
    \end{align*}
    where we use the trace estimate (see Lemma \ref{app: trace theory}) for the third step and $C_5:=C_{tr}^{\frac 32}T^{1/4} \Vert \psi_1 \Vert_{B} \gamma^{1/4}$, where $\gamma$ is as defined in \eqref{cond: boundary Gammaeps}. Employing Young's inequality we obtain  
\begin{equation}\label{ineq: trace linfinty l4}
    \begin{aligned}
         \eps \int_0^s \int_{\Gamma_\eps} \big( \aeps \big)^3 \psi_1\left(t,x,\frac{x}{\eps} \right)  \leq C_6 + \frac{\min\{1,2d_1\}}{8d_1(2+T)}\left \Vert \big( \aeps \big)^2 \right \Vert_{\mathrm L^2(0,T;\mathrm H^1(\Omega_{r_{\eps}}))}^{2},
 \end{aligned}
\end{equation}
where $C_6$ is a positive constant independent of $\eps$. From \eqref{ineq: l4 updated deciding ineq} and \eqref{ineq: trace linfinty l4} we obtain the following 
\begin{equation}\label{ineq: l4 ultimate updated deciding ineq}
    \begin{aligned}
    \sup_{0 \le t \le T} \int_{\Omega_{r_{\eps}}} \big( \aeps \big)^4(t,\cdot) 
    + &\frac{ \min\{1,2d_1\} }{2} \left \Vert \big( \aeps \big)^2 \right \Vert_{\mathrm L^2((0,T);\mathrm H^1(\Omega_{r_{\eps}}))}^2
    \\
    & \le (2+T)\int_{\Omega_{r_{\eps}}} \big( a^1_{\eps,0} \big)^4 
    + 4(2+T) \int_0^T \int_{\Omega_{r_{\eps}}} \ceps \big( \aeps \big)^3
    + C_6
    \\
    & \le C_6+4(2+T)(C^I)^4 + 4(2+T) \int_0^T \int_{\Omega_{r_{\eps}}} \ceps \big( \aeps \big)^3,
\end{aligned}
\end{equation}
where we have used the assumption \eqref{CI}. Employing H\"older's inequality we obtain
\begin{align}\label{L4 holder}
    \int_0^T \int_{\Omega_{r_{\eps}}} \ceps \big( \aeps \big)^3
     \leq \left \Vert \ceps \right \Vert_{\mathrm L^{10/3}((0,T);\mathrm L^{10/3}(\Omega_{r_{\eps}}))} \left \Vert \big( \aeps \big)^2 \right \Vert_{\mathrm L^{15/7}((0,T);\mathrm L^{15/7}(\Omega_{r_{\eps}}))}^{3/2}.
\end{align}
Invoking Anisotropic Sobolev inequality as described in Lemma \ref{app: anisotropic sobolev ineq} for $q=r=\frac{10}3$ yields
\begin{align*}
    \left \Vert a^i_{\eps} \right \Vert_{\mathrm L^{10/3}((0,T);\mathrm L^{10/3}(\Omega_{r_{\eps}}))} \leq  C_{AS} \left( \Vert a^i_{\eps} \Vert_{\mathrm L^{\infty}((0,T);\mathrm L^2(\Omega_{r_{\eps}}))} + \Vert \nabla a^i_{\eps} \Vert_{\mathrm L^2((0,T);\mathrm L^2(\Omega_{r_{\eps}}))} \right)^{\frac 12}, \ \forall\, i=1,2,3. 
\end{align*}
Thanks to Lemma \ref{anisotropic norm estimate}, we have that
\begin{align}\label{anisotropic L4}
    \left \Vert a^i_{\eps} \right \Vert_{\mathrm L^{10/3}((0,T);\mathrm L^{10/3}(\Omega_{r_{\eps}}))} \leq C_{AS}\kappa_2^{\frac 12}, \ \forall\, i=1,2,3. 
\end{align}
Using the above estimate \eqref{L4 holder}, it follows
\begin{equation*}
    \begin{aligned}
       \int_0^T \int_{\Omega_{r_{\eps}}} \ceps \big( \aeps \big)^3
        \leq C_{AS}\kappa_2^{\frac 12} \left \Vert \big( \aeps \big)^2 \right \Vert_{\mathrm L^{15/7}((0,T);\mathrm L^{15/7}(\Omega_{r_{\eps}}))}^{3/2} \leq C_7 \left \Vert (\aeps)^2 \right \Vert_{\mathrm L^{10/3}((0,T);\mathrm L^{10/3}(\Omega_{r_{\eps}}))}^{3/2},
    \end{aligned}
\end{equation*}
where we have used the H\"older inequality for the last step and $C_7$ is some positive constant, independent of $\eps$.  Hence from \eqref{ineq: l4 ultimate updated deciding ineq} we conclude that
\begin{equation}\label{ineq: l4 juicy}
\begin{aligned}
    \sup_{0 \le t \le T} \int_{\Omega_{r_{\eps}}} \big( \aeps \big)^4(t,\cdot) 
    + \frac{ \min\{1,2d_1\} }{2} &\left \Vert \big( \aeps \big)^2 \right \Vert_{\mathrm L^2((0,T);\mathrm H^1(\Omega_{r_{\eps}}))} \\
    \leq &C_8 + 4C_7(2+T)\left \Vert (\aeps)^2 \right \Vert_{\mathrm L^{10/3}((0,T);\mathrm L^{10/3}(\Omega_{r_{\eps}}))}^{3/2},
    \end{aligned}
\end{equation}
where $C_8:=C_6+4(2+T)(C^I)^4$. Employing Anisotropic Sobolev embedding (see Lemma \ref{app: anisotropic sobolev ineq}) in \eqref{ineq: l4 juicy}, we obtain
\begin{equation*}
    \begin{aligned}
        \left \Vert \big( \aeps \big)^2 \right \Vert_{\mathrm L^{10/3}((0,T);\mathrm L^{10/3}(\Omega_{r_{\eps}}))}^2 
        & \leq C_{AS} \left(\sup_{0 \le t \le T} \int_{\Omega_{r_{\eps}}} \big( \aeps \big)^4(t,\cdot) 
    +\left \Vert \nabla \big(\aeps \big)^2 \right \Vert^2_{\mathrm L^2((0,T);\mathrm L^2(\Omega_{r_{\eps}}))} \right)
        \\
        & \leq C_{AS}C_8 + 4C_{AS}C_7(2+T)\left \Vert (\aeps)^2 \right \Vert_{\mathrm L^{10/3}((0,T);\mathrm L^{10/3}(\Omega_{r_{\eps}}))}^{3/2}
        \\
        & \leq C_9 + \frac{\left \Vert (\aeps)^2 \right \Vert_{\mathrm L^{10/3}((0,T);\mathrm L^{10/3}(\Omega_{r_{\eps}}))}^{2}}{2}\, ,
    \end{aligned}
\end{equation*}
where the second step comes from the estimate \eqref{ineq: l4 juicy} and the last step is due to Young's inequality, where $C_9$ is some positive constant independent of $\eps$. Hence we get the following 
\begin{equation}\label{ineq: l4 half juicy}
\left \Vert \big( \aeps \big)^2 \right \Vert_{\mathrm L^{10/3}((0,T);\mathrm L^{10/3}(\Omega_{r_{\eps}}))}  \leq (2C_9)^{\frac 12}.
\end{equation}
Combining \eqref{ineq: l4 half juicy} with \eqref{ineq: l4 juicy} we obtain the desired result for $i=1$.  Similar calculations hold true for $i=2$ as well.
\end{proof}
The above estimate helps us deducing the following proposition. 
\begin{prop}\label{L2 integral estimate time derivative}
     Let $a^1_{\eps}, \beps $ and $ \aeps $ be the solution to the system \eqref{eq: aeps}-\eqref{eq: ceps} respectively. Then the following holds: 
    \begin{equation}\label{uniform: time Der L2L2 and grad LinftyL2}
        \Vert \dert a^i_{\eps} \Vert_{\mathrm L^{2}\big((0,T);\mathrm L^2(\Omega_{r_{\eps}})\big)} 
        + \Vert \nabla a^i_{\eps} \Vert_{\mathrm L^{\infty}\big((0,T);\mathrm L^2(\Omega_{r_{\eps}})\big)} \leq \kappa_4 \qquad \text{ for } 1 \le i \le 3,
    \end{equation}
    where $\kappa_4$ is a positive constant independent of $\eps$.
\end{prop}
\begin{proof}
    Testing the equation \eqref{eq: aeps} with $\dert \aeps$ on $(0,s) \times \Omega_{r_{\eps}}$ for $0 \le s \le T$ yields 
    \begin{equation*}
        \int_0^s \int_{\Omega_{r_{\eps}}} \big( \dert \aeps \big)^2 
        - d_1 \int_0^s \int_{\Omega_{r_{\eps}}} \Delta \aeps  \dert \aeps
        = \int_0^s \int_{\Omega_{r_{\eps}}} (\ceps - \aeps \beps) \dert \aeps.
    \end{equation*}
    Integrating by parts yields
    \begin{equation*}
        \int_0^s \int_{\Omega_{r_{\eps}}} \big( \dert \aeps \big)^2 
        + d_1 \int_0^s \int_{\Omega_{r_{\eps}}} \nabla \aeps \nabla \dert \aeps
        = \eps d_1 \int_0^s \int_{\Gamma_\eps} \psi_1 \left( t,x,\frac{x}{\eps} \right) \dert \aeps
        + \int_0^s \int_{\Omega_{r_{\eps}}} (\ceps - \aeps \beps) \dert \aeps.
    \end{equation*}
    Employing Young's inequality on the last term and Integrating by parts with respect to time variable in the boundary integral gives
    \begin{equation*}
    \begin{aligned}
        \int_0^s \int_{\Omega_{r_{\eps}}}& \big( \dert \aeps \big)^2 
        + \frac{d_1}{2} \int_0^s \int_{\Omega_{r_{\eps}}} \dert  | \nabla \aeps |^2
        \\
        & \leq \eps d_1 \int_{\Gamma_\eps} \aeps( s,\cdot ) \psi_1 \left( s,\cdot,\frac{\cdot}{\eps} \right) 
        - \eps d_1 \int_{\Gamma_\eps} \aeps( 0,\cdot ) \psi_1 \left( 0,\cdot,\frac{\cdot}{\eps} \right) 
        - \eps d_1 \int_0^s \int_{\Gamma_\eps} \aeps \dert \psi_1 \left( t,\cdot,\frac{\cdot}{\eps} \right)
        \\
        & \qquad \quad + \frac{1}{2}\int_0^s \int_{\Omega_{r_{\eps}}} | \dert \aeps |^2
        + \frac{1}{2}\int_0^s \int_{\Omega_{r_{\eps}}} \big| \ceps - \aeps \beps\big|^2 
        \\
        & \leq \eps d_1 \int_{\Gamma_\eps} \aeps( s,\cdot ) \psi_1 \left( s,\cdot,\frac{\cdot}{\eps} \right) 
        - \eps d_1 \int_0^s \int_{\Gamma_\eps} \aeps \dert \psi_1 \left( t,\cdot,\frac{\cdot}{\eps} \right) 
        + \frac{1}{2}\int_0^s \int_{\Omega_{r_{\eps}}} | \dert \aeps |^2
        \\
        & \qquad \quad + \int_0^s \int_{\Omega_{r_{\eps}}} | \ceps |^2 + \int_0^s \int_{\Omega_{r_{\eps}}} \big| \aeps \beps\big|^2, 
        \end{aligned}
    \end{equation*}
    where we have used the fact that $\psi\left( 0,x, \frac{x}{\eps} \right)=0$. Integrating the second term in the right with respect to time gives us 
    \begin{align*}
       \frac{1}{2}\int_0^s \int_{\Omega_{r_{\eps}}} \big( \dert \aeps \big)^2 
       & + \frac{d_1}{2} \int_{\Omega_{r_{\eps}}} | \nabla \aeps(s,\cdot) |^2 
       \\
       & \leq \eps d_1 \int_{\Gamma_\eps} \aeps( s,\cdot ) \psi_1 \left( s,\cdot,\frac{\cdot}{\eps} \right) 
       - \eps d_1 \int_0^s \int_{\Gamma_\eps} \aeps \dert \psi_1 \left( t,\cdot,\frac{\cdot}{\eps} \right) 
        \\
        & \qquad \quad + \int_0^s \int_{\Omega_{r_{\eps}}} | \ceps |^2 
        + \int_0^s \int_{\Omega_{r_{\eps}}} \big| \aeps \beps\big|^2 
        + \frac{d_1}{2} \int_{\Omega_{r_{\eps}}} | \nabla \aeps(0,\cdot) |^2 
        \\
        & \leq C^I+\kappa_1^2 + \eps d_1 \int_{\Gamma_\eps} \aeps( s,\cdot ) \psi_1 \left( s,\cdot,\frac{\cdot}{\eps} \right) 
        - \eps d_1 \int_0^s \int_{\Gamma_\eps} \aeps \dert \psi_1 \left( t,\cdot,\frac{\cdot}{\eps} \right) 
        \\
        & \qquad \quad + \left( \int_0^s \int_{\Omega_{r_{\eps}}} \big( \aeps \big)^4 \right)^{1/2} \left( \int_0^s \int_{\Omega_{r_{\eps}}} \big( \beps \big)^4 \right)^{1/2}, 
        \end{align*}
where we have used assumption \eqref{CI} and Lemma \ref{L2 integral estimate}. Thanks to Proposition \ref{L4 norm estimate}, the above estimate can be written as
        \begin{equation}\label{ineq: time l2 juicy}
    \begin{aligned}
        \frac{1}{2}\int_0^s \int_{\Omega_{r_{\eps}}} &\big( \dert \aeps \big)^2 
        + \frac{d_1}{2} \int_{\Omega_{r_{\eps}}} | \nabla \aeps(s,\cdot) |^2\\
        &\leq C_{10} + \eps d_1 \int_{\Gamma_\eps} \aeps( s,\cdot ) \psi_1 \left( s,\cdot,\frac{\cdot}{\eps} \right) 
        - \eps d_1 \int_0^s \int_{\Gamma_\eps} \aeps \dert \psi_1 \left( t,x,\frac{x}{\eps} \right) \, 
        \\
        & \leq C_{10}+ \eps d_1 \Vert \psi \Vert_{B} \int_{\Gamma_\eps} \aeps( s,\cdot )
        + \eps d_1 \Vert \psi \Vert_{B} \int_0^s \int_{\Gamma_\eps} \aeps
        \\
        & \leq C_{10} + \eps d_1 \Vert \psi \Vert_{B} \sqrt{|\Gamma_{\eps}|} \Vert \aeps(s,\cdot) \Vert_{\mathrm L^2(\Gamma_{\eps})} 
        + \eps d_1 \Vert \psi \Vert_{B} \sqrt{ s |\Gamma_{\eps}|} \Vert \aeps \Vert_{\mathrm L^2(0,s; \mathrm L^2(\Gamma_{\eps}))} \\
        & \leq C_{10}+\mathcal{J}_1+\mathcal{J}_2.
    \end{aligned}
    \end{equation}
    where  $C_{10}:=C^I+\kappa_1^2+T\kappa_3$. Using Lemma \ref{app: trace theory} we bound $\mathcal{J}_2$ as
    \begin{align*}\label{J2 calc}
        \mathcal{J}_2 \leq \sqrt{\eps} d_1 \Vert \psi \Vert_{B} s^{1/2} \sqrt{\eps |\Gamma_{\eps}| } C_{tr}^{\frac 12} \left[ \frac{1}{\eps} \int_0^s \int_{\Omega_{r_{\eps}}} \big|\aeps\big|^2  + \eps \int_0^s \int_{\Omega_{r_{\eps}}} \big|\nabla \aeps \big|^2  \right]^{1/2}. 
    \end{align*}
    Invoking Lemma \ref{L2 integral estimate} and Lemma \ref{anisotropic norm estimate} yields
    \begin{align}\label{J2 calc}
        \mathcal{J}_2 \leq  d_1 \Vert \psi \Vert_{B} T^{1/2} \gamma^{\frac 12} C_{tr}^{\frac 12}\left(\kappa_1^{\frac 12}+\kappa_2^{\frac 12}\right)=: C_{11},
    \end{align}
    where $\gamma$ is defined in \eqref{cond: boundary Gammaeps}. Using Lemma \ref{app: trace theory} we arrive at the
following bound:
    \begin{align*}
         \mathcal{J}_1 & \leq \sqrt{\eps} d_1 \Vert \psi \Vert_{B} \sqrt{\eps |\Gamma_{\eps}| } C_{tr}^{\frac 12} \left[ \frac{1}{\eps} \int_{\Omega_{r_{\eps}}} \big|\aeps(s,\cdot)\big|^2 + \eps \int_{\Omega_{r_{\eps}}} \big|\nabla \aeps(s,\cdot)\big|^2 \right]^{1/2}
         \\
         & \leq \sqrt{\eps} d_1 \Vert \psi \Vert_{B} \sqrt{\eps |\Gamma_{\eps}| } C_{tr}^{\frac 12} \left[ \sqrt{ \frac{1}{\eps} \int_{\Omega_{r_{\eps}}} \big|\aeps(s,\cdot)\big|^2 } + \sqrt{ \eps \int_{\Omega_{r_{\eps}}} \big|\nabla \aeps(s,\cdot)\big|^2 } \right]
         \\
         & \leq d_1\|\psi\|_{B}\gamma^{\frac 12} C_{tr}^{\frac 12} \kappa_2^{\frac 12} + d_1\|\psi\|_{B}\gamma^{\frac 12} C_{tr}^{\frac 12} \left\Vert \nabla \aeps(s,\cdot) \right\Vert_{\mathrm L^2(\Omega_{r_{\eps}})},
        \end{align*}
    where the last step is due to Lemma \ref{anisotropic norm estimate} and \eqref{cond: boundary Gammaeps}. Thanks to Young's inequality we obtain
    \begin{equation}\label{ineq: time l2 mid juicy}
        \begin{aligned}
       \mathcal{J}_1 \leq C_{12} + \frac{d_1}{4} \int_{\Omega_{r_{\eps}}} \big|\nabla \aeps(s,\cdot)\big|^2,
    \end{aligned}
    \end{equation}   
    where $C_{12}$ is a positive constant, independent of $\eps$. Combining the estimates \eqref{ineq: time l2 juicy}, \eqref{J2 calc} and \eqref{ineq: time l2 mid juicy} we obtain
    \begin{equation*}
        \frac{1}{2}\int_0^s \int_{\Omega_{r_{\eps}}} \big( \dert \aeps \big)^2 
       + \frac{d_1}{4} \int_{\Omega_{r_{\eps}}} | \nabla \aeps(s,\cdot) |^2
       \leq C_{10}+C_{11}+C_{12}. 
    \end{equation*}
   Since the above inequality is true for a.e. $0 \le s \le T$, we get the desired result.
   Similar calculation holds true for $\beps$ and $\ceps$.
\end{proof}

\section{Homogenization for $\mathcal{O}(\eps)$ perforations}\label{sec: Homogenization eps}
In this section we homogenize the system \eqref{eq: aeps}-\eqref{eq: ceps}. For the time being let us consider the equation \eqref{eq: aeps} and similar calculations holds for \eqref{eq: beps} and \eqref{eq: ceps} as well. For $1\leq i\leq 3$, let $\Tilde{a^i_{\eps}}, \Tilde{\nabla a^i_{\eps}}$, $\Tilde{\dert a^i_{\eps}}$ denote the extension by zero outside $\Omega_{r_{\eps}}$ of the functions $a^i_{\eps}, \nabla a^i_{\eps}, \partial_t a^i_{\eps}$ respectively. Before proceeding for the homogenization, observe that $\Tilde{a^i_{\eps}}, \Tilde{\nabla a^i_{\eps}}$ and $\Tilde{\dert a^i_{\eps}}$ are uniformly bounded and hence up to a subsequence two-scale converge to $\chi_{Y^*}(y)a^i(t,x), \, \, \chi_{Y^*}(y)[\nabla a^i(t,x) + \nabla_y a^{i,1}(t,x,y)]$ and $\chi_{Y^*}(y) \dert a^i(t,x)$ respectively, where $a^i \in \mathrm L^2((0,T);\mathrm H^1(\Omega))$ and $a^{i,1} \in L^2((0,T)\times \Omega; \mathrm H_{\rm per}^1(\Omega)/\RR )$ (see Lemma \ref{app: 2scale} and \ref{app: 2scale grad}).
Let us define the following function 
\begin{equation}\label{def: Feps}
F_{\eps}(t,x) := \ceps(t,x) - \aeps(t,x) \beps(t,x) \qquad t \in (0,T),\, x \in \Omega_{r_{\eps}}
\end{equation}
Thanks to the Lemma \ref{L2 integral estimate} and Proposition \ref{L4 norm estimate}, we observe that 
\[
    \Vert F_{\eps} \Vert_{\mathrm L^2((0,T) \times \Omega_{r_{\eps}})} \leq \kappa_1+2^{\frac 12}T^{\frac{1}{2}} \kappa_3^{\frac 12}. 
\]
Let us now state the following Lemma regarding the two-scale convergence of $F_{\eps}$.

\begin{lemma}
    Let $F_{\eps}$ be as defined in \eqref{def: Feps}. Then $\Tilde{F_{\eps}}$ $($extension by zero to $\Omega)$ two-scale converges to $F_h$, defined as  
    \begin{equation}\label{def: Fhom}
        F_h(t,x,y) := \chi_{Y^*}(y)\big(a^3(t,x)-a^1(t,x)a^2(t,x)\big), \qquad y \in Y,\, t \in (0,T),\, x \in \Omega,
    \end{equation}
    where $\chi_{Y^*}$ is the characteristic function of $Y^*$ \eqref{area outside the hole in unit cube}.
\end{lemma}
\begin{proof}
Define the following function 
\[
v^i_{\eps} := E_{\eps} a^i_{\eps} \qquad \forall 1 \leq i \leq 3,
\]
where $E_{\eps}$ is the extension operator as described in Lemma \ref{extension: desvilletes}. Thanks to Lemma \ref{L2 integral estimate}, Lemma \ref{anisotropic norm estimate} and Lemma \ref{extension: desvilletes}, we observe that 
\[
\| v^i_\eps \|_{\mathrm L^2((0,T);\mathrm{H}^1(\Omega))} 
\le \| a^i_\eps \|_{\mathrm L^2((0,T);\mathrm{H}^1(\Omega_{r_{\eps}}))} 
\le C_E (\kappa_1+\kappa_2).
\]
Also thanks to Theorem \ref{extension: time derivative theorem} and Proposition \ref{L2 integral estimate time derivative}, we have
\[
\| \partial_t v^i_\eps \|_{\mathrm L^2((0,T);\mathrm{L}^2(\Omega))} 
\le \| \partial_t a^i_\eps \|_{\mathrm L^2((0,T);\mathrm{L}^2(\Omega_{r_{\eps}}))} 
\le C_g \kappa_4.
\]
Using Aubin-Lions Lemma we obtain
\[
v_{\eps}^i \rightarrow v_i \qquad \text{in } \mathrm L^2((0,T);\mathrm L^2(\Omega)) 
\]
up to a subsequence, still indexed by $\eps$. Let $\phi \in C^1([0, T] \times \overline{\Omega}; C^\infty_\#(Y))$. Then, for $1\leq i\leq 3$, we obtain(see Lemma \ref{app: 2scale product}) 
\begin{align*}
    \int\limits_0^T \int\limits_{\Omega} \Tilde{a^i_{\eps}} \phi \left( t,x,\frac{x}{\eps} \right)
    =& \int\limits_0^T \int\limits_{\Omega_{r_{\eps}}} a^i_{\eps} \phi \left( t,x,\frac{x}{\eps} \right)
     = \int\limits_0^T \int\limits_{\Omega_{r_{\eps}}} v^i_{\eps} \phi \left( t,x,\frac{x}{\eps} \right)
     = \int\limits_0^T \int\limits_{\Omega} \Tilde{1} v^i_{\eps}  \phi \left( t,x,\frac{x}{\eps} \right)
     \\
     & \to \int\limits_0^T \int\limits_{\Omega} \int\limits_{Y} \chi_{Y^*}(y) v_i(t,x)  \phi \left( t,x,y \right) \, dt \,dx \, dy.
\end{align*}
where $\Tilde{1}$ is extension of the constant function $1$ by zero outside $\Omega_{r_{\eps}}$. As a result we obtain that 
\begin{equation}\label{two scale limit: a_i}
\Tilde{ a_{\eps}^i } \xrightarrow{2-s} \chi_{Y^*}(y) v_i(t,x).
\end{equation}
By the uniqueness of two-scale limit we observe that 
\begin{equation}\label{two-scale limit: equality}
    v_i(t,x) = a^i(t,x) \qquad \text{ for a.e.}\,  x \in \Omega \text{ and } t \in (0,T) . 
\end{equation}
Let $\phi \in C^1([0, T] \times \overline{\Omega}; C^\infty_\#(Y))$. Then
 \begin{align*}
        \int\limits_0^T \int\limits_{\Omega_{r_{\eps}}} F_\eps \phi \left( t,x,\frac{x}{\eps} \right)
        & = \int\limits_0^T \int\limits_{\Omega_{r_{\eps}}} \ceps \phi \left( t,x,\frac{x}{\eps} \right) 
        - \int\limits_0^T \int\limits_{\Omega_{r_{\eps}}} \aeps\beps \phi \left( t,x,\frac{x}{\eps} \right) 
        \\
        & = \int\limits_0^T \int\limits_{\Omega} \Tilde{\ceps} \phi \left( t,x,\frac{x}{\eps} \right)
        - \int\limits_0^T \int\limits_{\Omega} \Tilde{\aeps} v^2_{\eps} \phi \left( t,x,\frac{x}{\eps} \right) 
        \\
        & \to \int\limits_0^T \int\limits_{\Omega} \int\limits_{Y} \chi_{Y^*}(y) \big(a^3(t,x) - a^1(t,x) a^2(t,x)\big) \phi( t,x,y ). 
    \end{align*}
Hence proved.
\end{proof}
Recall the equation for $\aeps$
\begin{equation}\label{eq: aeps1}
    \left\{
    \begin{aligned}
        & \partial_t a^1_{\varepsilon} - d_1 \Delta a^1_{\varepsilon} = F_{\eps} \qquad && \text{in} \ (0,T) \times \Omega_{r_{\eps}} \\
        & \nabla a^1_{\varepsilon} \cdot n_{\varepsilon} = \varepsilon \psi_1\left(t,x, \frac{x}{\varepsilon} \right) \qquad && \text{on}\  (0,T) \times \Gamma_{\varepsilon} \\
        & \nabla a^1_{\varepsilon} \cdot n = 0 \qquad && \text{on}\ (0,T) \times \partial \Omega \\
        & a^1_{\varepsilon}(0,x) = a^{1}_{\eps,0}(x) \qquad && \text{in} \ \Omega_{r_{\eps}}.
    \end{aligned}
    \right.
\end{equation}
Define the following function 
\begin{equation}\label{test function: admissible}
\phi_{\eps}(t,x):= \phi_{0}(t,x) + \eps \phi_1 \left( t,x,\frac{x}{\eps} \right),  \end{equation}
where $\phi_0 \in \mathrm C^1([0,T] \times \overline{\Omega})$ and $\phi_1 \in C^1([0,T] \times \overline{\Omega}; \mathrm C^{\infty}_{\rm per}(Y)) $.
Testing \eqref{eq: aeps1} with $\phi_{\eps}$ as defined in \eqref{test function: admissible} and Integration by parts yields
\begin{equation}\label{eq: hom}
    \int\limits_0^T \int\limits_{\Omega_{r_{\eps}}} \dert \aeps \phi_{\eps} 
    + d_1 \int\limits_0^T \int\limits_{\Omega_{r_{\eps}}} \nabla \aeps \nabla \phi_{\eps}
    = \int\limits_0^T \int\limits_{\Omega_{r_{\eps}}} F_{\eps} \phi_{\eps}
    + \eps d_1 \int\limits_0^T \int\limits_{\Gamma_{\eps}} \psi_1 \left( t,\cdot, \frac{\cdot}{\eps} \right) \phi_{\eps}.
\end{equation}
Let us look at the left hand side 
\begin{equation}
    \begin{aligned}
    \int\limits_0^T \int\limits_{\Omega_{r_{\eps}}} \dert \aeps \phi_{\eps} 
    &+ d_1 \int\limits_0^T \int\limits_{\Omega_{r_{\eps}}} \nabla \aeps \nabla \phi_{\eps}
    \\
    & = \int\limits_0^T \int\limits_{\Omega} \Tilde{\dert \aeps} \phi_{\eps} 
    + d_1 \int\limits_0^T \int\limits_{\Omega} \Tilde{\nabla \aeps} \nabla \phi_{\eps}
    \\
    & \rightarrow \int\limits_0^T \int\limits_{\Omega} \int\limits_{Y} \chi_{Y^*}(y) \dert a^1 \phi_{0} 
    + d_1 \int\limits_0^T \int\limits_{\Omega} \int\limits_{Y} \chi_{Y^*}(y) [\nabla a^1 + \nabla_y a^{1,1}] [\nabla \phi_0 + \nabla_y \phi_{1}].
    \end{aligned}
\end{equation}
Similarly for the right hand side 
\begin{equation}
    \begin{aligned}
    \int\limits_0^T \int\limits_{\Omega_{r_{\eps}}} F_{\eps} \phi_{\eps}
    + \eps d_1 \int\limits_0^T \int\limits_{\Omega_{r_{\eps}}} \psi_1 \left( t,\cdot, \frac{\cdot}{\eps} \right) \phi_{\eps}
    & = \int\limits_0^T \int\limits_{\Omega} \Tilde{F_{\eps}} \phi_{\eps}
    + \eps d_1 \int\limits_0^T \int\limits_{\Gamma_{\eps}} \psi_1 \left( t,\cdot, \frac{\cdot}{\eps} \right) \phi_{\eps}
    \\
    & \rightarrow \int\limits_0^T \int\limits_{\Omega} \int\limits_{Y} F_h \phi_{0} 
    + d_1 \int\limits_0^T \int\limits_{\Omega} \int\limits_{\Gamma} \psi_1(t,x,y) \phi_0, 
    \end{aligned}
\end{equation}
where we have used Proposition \ref{app: 2scale trace}. Thus we obtain the following variational formulation 
\begin{equation}
    \begin{aligned}
    \int\limits_0^T \int\limits_{\Omega} \int\limits_{Y} \chi_{Y^*}(y) \dert a^1 \phi_{0} 
    & + d_1 \int\limits_0^T \int\limits_{\Omega} \int\limits_{Y} \chi_{Y^*}(y) [\nabla a^1 + \nabla_y a^{1,1}] [\nabla \phi_0 + \nabla_y \phi_{1}] 
    \\
    &= \int\limits_0^T \int\limits_{\Omega} \int\limits_{Y^*} F_h \phi_{0} 
    + d_1 \int\limits_0^T \int\limits_{\Omega} \int\limits_{\Gamma} \psi_1(t,x,y) \phi_0.
    \end{aligned}
\end{equation}
which corresponds to the following homogenized system 
\begin{equation*}
    \left\{
    \begin{aligned}
         \theta \frac{\partial a^1}{\partial t}(t,x) 
        - & d_1 \dv \left[ \int\limits_{Y^*} [ \nabla a^1(t,x) + \nabla_y a^{1,1}(t,x,y)] \right] 
        \\& \qquad \qquad = \int\limits_{Y} F_h(t,x,y)
          + d_1 \int\limits_{\Gamma} \psi_1(t,x,y) 
        \qquad && \text{in }(0,T) \times \Omega
        \\
         &\int\limits_{Y^*} [ \nabla a^1(t,x) + \nabla_y a^{1,1}(t,x,y)] \cdot n  = 0 \qquad && \text{on } (0,T) \times \partial \Omega 
        \\
         &a^1(0,x)  = a^1_0(x) \qquad && \text{in } \Omega 
        \\
        & - d_1 \dv_y [\nabla a^1(t,x) + \nabla_y a^{1,1}(t,x,y)]  = 0 \qquad && \text{in } (0,T) \times \Omega \times Y^*
        \\
        & [\nabla a^1(t,x) + \nabla_y a^{1,1}(t,x,y)] \cdot n  = 0 \qquad && \text{on } (0,T) \times \Omega \times \Gamma,
    \end{aligned}
    \right.
\end{equation*}
where we have passed the two-scale limit for initial condition (see \cite[Section 3.2]{desvillettes2018homogenization}, \cite[Chapter 11]{MR1765047}) and $\theta = |Y^*|$, where $Y^*$ is as defined in \eqref{cond: boundary Gammaeps}. Using the linearity of the equation we can write $a^{1,1}$ as follows 
\[
a^{1,1}(t,x,y) = \omega_j(y) \frac{\partial a^1}{\partial x_j}(t,x),
\]
where $\omega_j$ solves the following cell problem
\begin{equation}
    \left\{
    \begin{aligned}
        & -\dv_y [ \nabla_y\omega_j + e_j] = 0 \qquad Y^*
        \\
        & [ \nabla_y\omega_j + e_j]\cdot n = 0 \qquad \Gamma
        \\
        & \omega_j \in \mathrm H^1(Y^*) \text{ and is Y-periodic}.
    \end{aligned}
    \right.
\end{equation}
Hence we obtain the following homogenized equation
\begin{equation}\label{eq: hom aeps}
    \left\{
    \begin{aligned}
        \theta \frac{\partial a^1}{\partial t}(t,x) 
        - &d_1 \dv \left[ D\nabla a^1(t,x) \right] \\
        =& \theta \big( a^3(t,x) - a^1(t,x) a^2(t,x)\big) + d_1 \int\limits_{\Gamma} \psi_1(t,x,y) 
        \qquad && \text{in} \ (0,T) \times \Omega
        \\
        & D\nabla a^1(t,x) \cdot n = 0 \qquad && \text{on} \ (0,T) \times \partial \Omega 
        \\
        & a^1(0,x) = a^1_0(x) \qquad && \text{in} \ \Omega, 
    \end{aligned}
    \right.
\end{equation}
where $D$ is a $3 \times 3$ matrix whose entries are given by 
\[
D_{ij}:= \int_{Y^*} \delta_{ij} + \frac{\partial \omega_j}{\partial y_i}.
\]
Similar calculations holds true for \eqref{eq: beps} and \eqref{eq: ceps} as well.
Thus we conclude the proof of Theorem \ref{main convergence theorem}.

\section{Homogenization for $\mathcal{O}(\eps^{\alpha})$ perforations}\label{sec: Homogenization rapid}

In this section, we homogenize the equation \eqref{eq: aeps}-\eqref{eq: ceps} for the case when the perforations are of the size $\mathcal{O}(\eps^{\alpha})$ for $\alpha>1$. Since all the bounds mentioned in Section \ref{sec: Uniform estimates} holds true for these perforations as well. We define the following function
\[
\mathfrak{v}_{\eps}^i(t,x):= E_{\eps}(a_{\eps}^i),
\]
and from the calculations done in previous section we deduce that 
\[
\mathfrak{v}_{\eps}^i \rightarrow \mathfrak{v}_i \qquad \text{in } \mathrm L^2((0,T);\mathrm L^2(\Omega)), 
\]
up to a subsequence, still indexed by $\eps$. Also observe that in the current regime of the holes we have the following strong convergence
\begin{align*}
    \chi_{\Omega_{r_{\eps}}} \to 1 , \ \text{strongly in }\mathrm{L}^p(\Omega) \text{ for } 1 \leq p < \infty.
\end{align*}
Hence the following strong convergence also holds:
\begin{equation}\label{eq: strong cgnc rapid}
    \Tilde{a_{\eps}^i} \to \mathfrak{v}_i \qquad \text{in } \mathrm L^2((0,T);\mathrm L^2(\Omega)). 
\end{equation}
Testing the equation \eqref{eq: aeps} with $\phi \in C^{1}([0,T] ;C^{\infty}( \overline{\Omega}))$ such that $\phi(T,\cdot) = 0$ results in
\begin{equation}\label{eq: variational formulation rapid}
    \int\limits_0^T \int\limits_{\Omega_{r_{\eps}}} \dert \aeps \phi 
    + d_1 \int\limits_0^T \int\limits_{\Omega_{r_{\eps}}} \nabla \aeps \cdot \nabla \phi
    - d_1 \eps \int\limits_0^T \int\limits_{\Gamma_{\eps}} \psi_1\left(t,\cdot, \frac{\cdot}{\eps} \right) \phi
    = \int\limits_0^T \int\limits_{\Omega_{r_{\eps}}} \ceps \phi 
    - \int\limits_0^T \int\limits_{\Omega_{r_{\eps}}} \aeps \beps \phi 
\end{equation}
Consider the boundary integral
\begin{align*}
\eps  \int\limits_0^T \int\limits_{\Gamma_{\eps}} \psi_1\left(t,\cdot, \frac{\cdot}{\eps} \right) \phi
& \leq \eps \| \psi \|_{B} \| \phi \|_{\mathrm L^{\infty}((0,T)\times \overline{\Omega})} T |\Gamma_{\eps}|
\to 0
\end{align*}
where we have used the fact that $\eps |\Gamma_{\eps}| \to 0$ thanks to the size of the perforations. Integrating by parts with respect to the time derivative in the first term in equation \eqref{eq: variational formulation rapid} results in 
\begin{align*}
\int_{0}^{T} \int_{\Omega_{r_{\eps}}} \partial_{t} \aeps \cdot \phi 
&= -\int_{0}^{T} \int_{\Omega_{r_{\eps}}} \aeps  \partial_{t} \phi 
+ \int_{\Omega_{r_{\eps}}} \aeps(T, \cdot) \phi(T, \cdot) 
 - \int_{\Omega_{r_{\eps}}} \aeps(0, \cdot) \phi(0, \cdot) 
\\
&= -\int_{0}^{T} \int_{\Omega_{r_{\eps}}} \aeps \partial_{t} \phi - \int_{\Omega_{r_{\eps}}} \aeps(0, \cdot) \phi(0, \cdot) 
\\
&= -\int_{0}^{T} \int_{\Omega} \Tilde{\aeps} \partial_{t} \phi 
- \int_{\Omega_{r_{\eps}}} \aeps(0, \cdot) \phi(0, \cdot) 
\\
& \rightarrow -\int_{0}^{T} \int_{\Omega} \mathfrak{v}^1 \partial_{t} \phi 
- \int_{\Omega} a_{0} \cdot \phi(0, \cdot)
\end{align*}
where we have used equation \eqref{eq: strong cgnc rapid} and the convergence of $\aeps(0, \cdot)$ to $a_0$. Let us now look at the second integral in the left hand side of equation \eqref{eq: variational formulation rapid}
\begin{align*}
    \int\limits_0^T \int\limits_{\Omega_{r_{\eps}}} \nabla \aeps \cdot \nabla \phi
     = \int\limits_0^T \int\limits_{\Omega_{r_{\eps}}} \nabla E_{\eps} (\aeps) \cdot \nabla \phi
    & = \int\limits_0^T \int\limits_{\Omega} \chi_{\Omega_{r_{\eps}}}\nabla E_{\eps} (\aeps) \cdot \nabla \phi
    \\
    & \to \int\limits_0^T \int\limits_{\Omega} \nabla \mathfrak{v}_1 \cdot \nabla \phi 
\end{align*}
Let us now consider the first term in the right hand side of the equation \eqref{eq: variational formulation rapid} using the convergence as in \eqref{eq: strong cgnc rapid}
\begin{align*}
    \int\limits_0^T \int\limits_{\Omega_{r_{\eps}}} \ceps \phi 
     = \int\limits_0^T \int\limits_{\Omega} \Tilde{\ceps} \phi 
    \to  \int\limits_0^T \int\limits_{\Omega_{r_{\eps}}} \mathfrak{v}_3 \phi 
\end{align*}
Finally we consider the second term from the right hand side of the equation \eqref{eq: variational formulation rapid}
\begin{align*}
\int\limits_0^T \int\limits_{\Omega_{r_{\eps}}} \aeps \beps \phi 
& = \int\limits_0^T \int\limits_{\Omega} \Tilde{\aeps} (\beps- \mathfrak{v}_2) \phi + \int\limits_0^T \int\limits_{\Omega} \Tilde{\aeps} \mathfrak{v}_2 \phi 
\\
& =: \mathcal{V}_1 + \mathcal{V}_2 
\end{align*}
Observe that
\begin{equation*}
\begin{aligned}
   \mathcal{V}_1 
   & \leq \| \beps- \mathfrak{v}_2 \|_{\mathrm L^2((0,T) \times \Omega)} \| \aeps \phi\|_{\mathrm L^2((0,T) \times \Omega)} 
   \\
   & \leq \| \phi\|_{\mathrm L^{\infty}((0,T) \times \Omega)}  \| \beps- \mathfrak{v}_2 \|_{\mathrm L^2((0,T) \times \Omega)} \| \aeps \|_{\mathrm L^2((0,T) \times \Omega)} 
   \\
   & \to 0
   \end{aligned}
\end{equation*}
Similarly we get 
\begin{equation*}
\begin{aligned}
   \mathcal{V}_2
    = \int\limits_0^T \int\limits_{\Omega} \Tilde{\aeps} \mathfrak{v}_2 \phi 
   \to \int\limits_0^T \int\limits_{\Omega} \mathfrak{v}_1 \mathfrak{v}_2 \phi 
   \end{aligned}
\end{equation*}
Combining all the above estimates together we obtain the following variational formulation for equation \eqref{eq: strong cgnc rapid}
\begin{equation*}
    -\int_{0}^{T} \int_{\Omega} \mathfrak{v}^1 \partial_{t} \phi 
    - \int_{\Omega} a_{0} \cdot \phi(0, \cdot) 
    + d_1 \int\limits_0^T \int\limits_{\Omega} \nabla \mathfrak{v}_1 \cdot \nabla \phi
    = \int\limits_0^T \int\limits_{\Omega_{r_{\eps}}} \mathfrak{v}_3 \phi 
    - \int\limits_0^T \int\limits_{\Omega} \mathfrak{v}_1 \mathfrak{v}_2 \phi 
\end{equation*}
which suggests that $\mathfrak{v}_1$ is the weak solution of the following problem
\begin{equation*}
\left\{
    \begin{aligned}
        \dert \mathfrak{v}_1 
        - d_1 \Delta \mathfrak{v}_1 
        & = \mathfrak{v}_3 - \mathfrak{v}_1\mathfrak{v}_2 
        \qquad && \text{in } (0,T) \times \Omega
        \\
        \nabla \mathfrak{v}_1 \cdot n & = 0  \qquad && \text{in } (0,T) \times \partial \Omega
        \\
        \mathfrak{v}_1(0,\cdot) & = a_0 \qquad && \text{in } (0,T) \times \partial \Omega
    \end{aligned}
    \right.
\end{equation*}
Similar calculations holds true for \eqref{eq: beps} and \eqref{eq: ceps} as well. Hence we conclude the proof of Theorem \ref{theorem: rapid main}.


\appendix

\renewcommand{\thesection}{\Alph{section}}

\section{Existence Theory}\label{sec: appendix existence theory}
In this section, we analyze the existence of global-in-time classical nonnegative solution to the three species reaction-diffusion system in the perforated domain $\Omega_{r_\eps}$.
\begin{theorem}\label{existence, uniqueness,regularity, epsilon scale}
    For $i=1,2,3$, consider the system \eqref{eq: aeps}-\eqref{eq: ceps} with $\eps\ll 1$. Then there exists a unique global-in-time classical solution $a^i_{\eps}$ to the system \eqref{eq: aeps}-\eqref{eq: ceps}, satisfying
    \[
    a^i_{\eps}\in \mathrm{W}^{1,p}((0,T); \mathrm{L}^p(\Omega_{r_\eps})) \cap \mathrm{L}^p((0,T); \mathrm{W}^{2,p}(\Omega_{r_\eps})) \cap \mathrm{L}^{\infty}((0,T)\times\Omega_{r_\eps}), 
    \]
    for all $i=1,2,3$, $T\in(0,\infty)$, and $p\in(1,\infty)$ with $p\neq \frac 32,3$. Furthermore, the solution is interior smooth and 
    \[
    a_{\eps}^i \geq 0, \ \forall \, i=1,2,3.
    \]
\end{theorem}
\begin{proof}
   Combining the facts that
\[
a_{\varepsilon,0}^i \in \mathrm{W}^{2,p}(\Omega_{r_\varepsilon}) \quad \text{and} \quad \psi \in C^1([0,T], B),
\]
where 
\[
B := C^1\big(\overline{\Omega}; C^1_{\#}(Y)\big)
\]
(with $C^1_{\#}(Y)$ denoting the space of $Y$-periodic $C^1$ functions), together with the compatibility condition \eqref{compatibility condition}, we conclude the desired result by following \cite[Theorem 3.5]{pierre2010global} and \cite[Theorem 1]{prussexistence}. We would like to note that the nonnegativity of the solution comes from the quasipositive \eqref{quasipositive} structure of the source term (see \cite[Lemma 1.1]{pierre2010global} for further details).
\end{proof}
For the convenience, we state below the result concerning the nonnegativity of the solution.
\begin{lemma}[Nonnegativity of solutions, \cite{pierre2010global}]\label{positivity of solution}
For $1\le i\le m$, let $\mu_i>0$ be constant. Let $u_i:(0,T)\times\Omega\to\mathbb{R}$ be a classical solution of the system: for $1\le i\le m$
\[
\left\{
\begin{aligned}
\partial_t u_i - \mu_i \Delta u_i &= f_i(u_1,\dots,u_m) && \text{in } (0,T)\times\Omega\\
\nabla u_i\cdot n &= \beta(t,x)\geq 0 && \text{on } (0,T)\times\partial\Omega\\
u_i(0,x) &\ge 0 && \text{in } \Omega.
\end{aligned}
\right.
\]
Assume that $f=(f_1,\dots,f_m):\mathbb{R}^m\to\mathbb{R}^m$ is quasipositive, that is,
\begin{align}\label{quasipositive}
f_i(r_1,\dots,r_{i-1},0,r_{i+1},\dots,r_m) \ge 0,
\quad \forall (r_1,\dots,r_m)\in [0,+\infty)^m.
\end{align}
Then the solution remains nonnegative, namely $u_i\ge 0$ for all $1\le i\le m$.
\end{lemma}
\begin{remark}\label{trace realization}
    It follows that both $a_{\varepsilon}^i$ and $(a_{\varepsilon}^i)^2$ admit $\mathrm{L}^2$-traces on the boundary $\partial \Omega_{r_\varepsilon}$. Moreover, the time derivative $\partial_t a_{\varepsilon}^i$ also admits a boundary trace, which is understood in the sense of $\mathcal{D}'\big((0,T); \mathrm{W}^{-1/2,2}(\partial \Omega_{r_\varepsilon})\big)$. Hence all the regularity assumption we need for our analysis holds true.
\end{remark}

\section{Extension Operator and its time derivative}\label{sec: appendix extension operator}

In this section, we recall the existence of the following family of extension operators from \cite{franchi2016microscopic, donato_neumann_nonhomogeneous, MR2598093} and construct the time derivative of these operators under certain scenarios. 

\begin{lemma}[see \cite{MR2598093, donato_neumann_nonhomogeneous}]\label{extension: desvilletes}
There exists a family of linear continuous extension operators
\[
E_\eps : \mathrm H^{1}(\Omega_{r_\eps}) \to \mathrm{H}^{1}(\Omega)
\]
and a constant $C_E > 0$, independent of $\eps$, such that
\[
E_\eps v = v \quad \text{in } \Omega_{r_\eps},
\]
and
\begin{equation*}
\begin{aligned}
\|E_\eps v\|_{\mathrm L^{2}(\Omega)} 
& \le C_E \| v\|_{\mathrm L^{2}(\Omega_{r_\eps})},
\\
\|\nabla E_\eps v\|_{\mathrm L^{2}(\Omega)} 
& \le C_E \| \nabla v\|_{\mathrm L^{2}(\Omega_{r_\eps})},
\end{aligned}
\end{equation*}
for each $v \in \mathrm{H}^{1}(\Omega_{r_\eps})$.
\end{lemma}
\begin{remark}
If $v \in \mathrm L^2((0,T);\mathrm H^1(\Omega_{r_\eps}))$ then $E_{\eps}v \in \mathrm L^2((0,T);\mathrm H^1(\Omega))$.
\end{remark}
Using this extension operator we define a new operator as follows  
\[
\mathcal G_\eps : \mathrm L^{2}(\Omega_{r_\eps}) \to \mathrm{L}^{2}(\Omega)
\]
such that 
\begin{equation}\label{def: Geps}
\mathcal G_\eps\xi(x) = \lim_{n \to \infty} E_{\eps}\phi_n(x),
\end{equation}
where $\{\phi_n\}_{n\ge 1} \subset \mathrm H^{1}(\Omega_{r_\eps})$ is such that 
\[
\| \phi_n - \xi \|_{\mathrm L^2(\Omega_{r_\eps})} \to 0.
\]
Then we have the following Lemma for the operator $\mathcal G_{\eps}$.
\begin{lemma}\label{extension: sdks}
$\mathcal G_{\eps}$ as defined in \eqref{def: Geps} is a well-defined bounded linear operator such that for $v \in \mathrm{L}^{2}(\Omega_{r_\eps})$ 
\[
\mathcal G_{\eps} v = v \quad \text{in } \Omega_{r_\eps},
\]
and
\begin{equation*}
\|\mathcal G_{\eps} v\|_{\mathrm L^{2}(\Omega)} \le C_{g} \| v\|_{\mathrm L^{2}(\Omega_{r_\eps})},
\end{equation*}
where constant $C_{g} > 0$ is independent of $\eps$. Furthermore for $w \in \mathrm H^1(\Omega_{r_\eps})$ 
\begin{equation}\label{extension: Geps=Eeps}
\mathcal G_{\eps} w = E_{\eps} w.
\end{equation}
\end{lemma}
Using $\mathcal G_{\eps}$, we define the time derivative of $E_{\eps}$ in the following Theorem.
\begin{theorem}\label{extension: time derivative theorem}
    Suppose $u \in \mathrm L^2((0,T);\mathrm H^1(\Omega_{r_\eps}))$ and $\partial_t u \in \mathrm L^2((0,T);\mathrm L^2(\Omega_{r_\eps}))$. Then $ E_{\eps}u$ is weakly differentiable with respect to time and it's weak derivative is given by 
    \[
    \partial_t E_{\eps}u = \mathcal{G}_{\eps}(\partial_t u).
    \]
    Furthermore $\partial_t E_{\eps}u \in \mathrm L^2((0,T);\mathrm L^2(\Omega))$ and
    \begin{equation}\label{extension: time derivative bound}
        \|\partial_t E_\eps u\|_{\mathrm L^2((0,T);\mathrm L^2(\Omega))} \le C_g \| \partial_t u\|_{\mathrm L^2((0,T);\mathrm L^{2}(\Omega_{r_\eps}))},
    \end{equation}
    where constant $C_{g} > 0$ is as mentioned in Lemma \ref{extension: sdks}.
\end{theorem}
\begin{proof}
    Observe that $u \in \mathrm H^1((0,T);\mathrm L^2(\Omega_{r_\eps}))$ and hence by \cite[Proposition 9.3]{brezis2011functional} we have 
    \begin{equation}\label{brezis: prop 9.3}
    \lim_{h \to 0}\left\| \partial_t u - \frac{u(t+h,\cdot) - u(t,\cdot)}{h} \right\|_{\mathrm L^2((\delta,T-\delta);\mathrm L^2(\Omega_{r_\eps}))} = 0, \ \forall \, 0 <\delta < T/4.
    \end{equation}
    Let $\phi \in \mathrm C_c^{\infty}([t_1,t_2] \times \Omega_{r_\eps})$ be such that $0<t_1<t_2<T$. Testing $E_{\eps} u$ with $ \dert \phi$ gives us
    \begin{align*}
    \int_{0}^{T} \int_{\Omega} E_{\eps} u \cdot \partial_t\phi  
    &= \lim_{h \to 0} \int_{\frac{t_1}{2}}^{T - (\frac{T-t_2}{2})} \int_{\Omega_{r_\eps}} E_{\eps} u \frac{ \phi(t+h, \cdot) - \phi(t, \cdot)}{h} 
    \\
    &= \lim_{h \to 0} \left[ \int_{\frac{t_1}{2} + h}^{T - (\frac{T-t_2}{2}) + h} \int_{\Omega_{r_\eps}} \frac{E_{\eps}u(t-h, \cdot) \phi(t)}{h} - \int_{t_1}^{T 
    - (\frac{T-t_2}{2})} \int_{\Omega_{r_\eps}} \frac{E_{\eps}u(t, \cdot) \phi(t, \cdot)}{h} \, dx dt \right] 
    \\
    &= \lim_{h \to 0} \int_{t_1}^{t_2} \int_{\Omega_{r_\eps}} \frac{E_{\eps}u(t-h, \cdot) \phi(t, \cdot) - E_{\eps}u(t, \cdot) \phi(t, \cdot)}{h} 
    \\
    &= \lim_{h \to 0} \int_{t_1}^{t_2} \int_{\Omega_{r_\eps}} \mathcal{G}_{\eps} \left( \frac{u(t-h, \cdot) - u(t, \cdot)}{h} \right)  \phi(t, \cdot) 
    \\
    &= \int_{t_1}^{t_2} \int_{\Omega_{r_\eps}} \mathcal{G}_{\eps} \left( \lim_{h \to 0} \frac{u(t-h, \cdot) - u(t, \cdot)}{h} \right) \phi(t, \cdot) 
    \\
    &= -\int_{t_1}^{t_2} \int_{\Omega} \mathcal{G}_{\eps} \left( \frac{\partial u}{\partial t} \right)  \phi(t, \cdot),
\end{align*}
where we have used Dominated Convergence Theorem for first and fifth steps. Thus in the sense of distribution 
\[
\partial_t E_{\eps}(u) = \mathcal{G}_{\eps} \left( \partial_t u \right).
\]
Hence invoking Lemma \ref{extension: sdks} yields inequality \eqref{extension: time derivative bound}.
\end{proof}

\section{Trace Inequality and Anisotropic Sobolev Inequality}\label{sec: appendix trace theory and anisotropic sobolev}
In this section, we recall the Trace inequality (see \cite{AllaireDamlamianHornung1996, desvillettes2018homogenization, franchi2016microscopic, MR2598093, MR2908615}) and the Anisotropic Sobolev Inequality (see \cite{franchi2016microscopic, desvillettes2018homogenization}).

\begin{lemma}[see \cite{AllaireDamlamianHornung1996, MR2598093, MR2908615}]\label{app: trace theory} 
Let $r_{\eps} =  \mathcal{O}(\eps)$. Then there exists a constant $C_{tr} > 0$ which does not depend on $\eps$, such that when $v \in Lip(\Omega_{r_\eps})$, then
\begin{equation*}
\|v\|^2_{\mathrm L^2(\Gamma_{r_\eps})} \le C_{tr} \left[ \frac{1}{\eps} \int_{\Omega_{r_\eps}} |v|^2 \, dx + \eps \int_{\Omega_{r_\eps}} |\nabla_x v|^2 \, dx \right].
\end{equation*}
\end{lemma}

\begin{lemma}
[
see \cite{donato_neumann_nonhomogeneous}
]
\label{app: rapid trace theory} 
Let $r_{\eps} =  \mathcal{O}(\eps^{\alpha})$ for $\alpha>1$. Then there exists a constant $C_{tr} > 0$ which does not depend on $\eps$, such that when $v \in \mathrm{H}^1(\Omega_{r_\eps})$, then
\begin{equation*}
\|v\|^2_{\mathrm L^2(\Gamma_{r_\eps})}
\le C_{tr}
\left\{
\begin{aligned}
& \left( \frac{r_{\eps}}{\eps^{\frac{3}{2}}}\right)^{2} \|v\|^2_{\mathrm H^1(\Omega_{r_\eps})}
\qquad 1 <\alpha \leq 3,
\\
& r_{\eps} \|v\|^2_{\mathrm H^1(\Omega_{r_\eps})} \qquad \alpha >3.
\end{aligned}
\right.
\end{equation*}
\end{lemma}

\begin{lemma}[Anisotropic Sobolev inequalities in perforated domains, see \cite{franchi2016microscopic}]\label{app: anisotropic sobolev ineq}  
\begin{itemize}
\item[(i).] For $q_1$ and $r_1$ satisfying the conditions
\begin{equation}
\begin{cases}
\dfrac{1}{r_1} + \dfrac{3}{2q_1} = \dfrac{3}{4}, \\[10pt]
r_1 \in [2, \infty], \ q_1 \in [2, 6],
\end{cases} 
\end{equation}
the following estimate holds $($for $v \in \mathrm{ H}^1((0, T); L^2(\Omega_{r_\eps})) \cap L^2((0, T); \mathrm{ H}^1(\Omega_{r_\eps}))):$
\begin{equation}
\|v\|_{L^{r_1}((0, T); L^{q_1}(\Omega_{r_\eps}))} \le C_{AS} \|v\|_{Q_{r_\eps}(T)},
\end{equation}
where $C_{AS} > 0$ is independent of $\eps$ and
\begin{equation}
\|v\|_{Q_{r_\eps}(T)}^2 := \sup_{0 \le t \le T} \int_{\Omega_{r_\eps}} |v(t)|^2 \, dx + \int_0^T dt \int_{\Omega_{r_\eps}} |\nabla v(t)|^2 \, dx. 
\end{equation}
\end{itemize}
\end{lemma}
\begin{remark}
Proof of Anisotropic Sobolev Inequality depends upon the existence of Extension operator which exists thanks to Appendix \ref{sec: appendix extension operator}. Hence we get that Anisotropic Sobolev Inequality holds true for the case when $\alpha >1$.   
\end{remark}

\section{Two-Scale Convergence}\label{sec: appendix two-scale}
In this section, we gather some useful results on two-scale convergence (see \cite{allaire_2scale, AllaireDamlamianHornung1996, Nguetseng_2scale, nand_neumann, MR1670679, hornung1992applications})

\begin{defn}[see \cite{allaire_2scale}]
A sequence of functions $v^\eps$ in $\mathrm L^2((0, T); \mathrm L^2(\Omega))$ two-scale converges to $v_0 \in \mathrm L^2((0, T); \mathrm L^2(\Omega \times Y))$ if
\begin{equation*}
\lim_{\eps \to 0} \int_0^T \int_\Omega v^\eps(t, x) \, \phi\left(t, x, \frac{x}{\eps}\right) dt \, dx = \int_0^T \int_\Omega \int_Y v_0(t, x, y) \, \phi(t, x, y) \, dt \, dx \, dy
\end{equation*}
for all $\phi \in C^1([0, T] \times \overline{\Omega}; C^\infty_\#(Y))$.
\end{defn}

We recall the following classical Proposition:

\begin{prop} \label{app: 2scale}
If $v^\eps$ is a bounded sequence in $L^2((0, T) \times \Omega)$, then there exists a function $v_0 := v_0(t, x, y)$ in $L^2((0, T) \times \Omega \times Y)$ such that, up to a subsequence, $v^\eps$ two-scale converges to $v_0$.
\end{prop}

Then, the following Proposition is useful for obtaining the limit of the product of two two-scale convergent sequences.

\begin{prop} \label{app: 2scale product}
Let $v^\eps$ be a sequence of functions in $L^2((0, T) \times \Omega)$ which two-scale converges to a limit $v_0 \in L^2((0, T) \times \Omega \times Y)$. Suppose furthermore that
\begin{equation*}
\lim_{\eps \to 0} \int_0^T \int_\Omega |v^\eps(t, x)|^2 \, dt \, dx = \int_0^T \int_\Omega \int_Y |v_0(t, x, y)|^2 \, dt \, dx \, dy. 
\end{equation*}
Then, for any sequence $w^\eps$ in $L^2((0, T) \times \Omega)$ that two-scale converges to a limit $w_0 \in L^2((0, T) \times \Omega \times Y)$, we get the limit
\begin{align*}
\lim_{\eps \to 0} \int_0^T \int_\Omega v^\eps(t, x) \, w^\eps(t, x) &\, \phi\left(t, x, \frac{x}{\eps}\right) dt \, dx  \\
&= \int_0^T \int_\Omega \int_Y v_0(t, x, y) \, w_0(t, x, y) \, \phi(t, x, y) \, dt \, dx \, dy, 
\end{align*}
for all $\phi \in C^1([0, T] \times \overline{\Omega}; C^\infty_\#(Y))$.
\end{prop}

The next Propositions yield a characterization of the two-scale limits of gradients of bounded sequences $v^\eps$. This result is crucial for applications to homogenization problems.

\begin{prop}\label{app: 2scale grad} 
Let $v^\eps$ be a bounded sequence in $L^2((0, T); \mathrm{ H}^1(\Omega))$ that converges weakly to a limit $v := v(t, x)$ in $L^2((0, T); \mathrm{ H}^1(\Omega))$. Then, $v^\eps$ also two-scale converges to $v$, and there exists a function $v_1 := v_1(t, x, y)$ in $L^2((0, T) \times \Omega; \mathrm{ H}^1_\#(Y)/ \mathbb{R})$ such that, up to extraction of a subsequence, $\nabla v^\eps$ two-scale converges to $\nabla_x v(t, x) + \nabla_y v_1(t, x, y)$.
\end{prop}

The main result of two-scale convergence can be generalized to the case of sequences defined in $L^2([0, T] \times \Gamma_\eps)$.

\begin{prop}[see \cite{allaire_2scale, AllaireDamlamianHornung1996}]\label{app: 2scale trace}
Let $v^\eps$ be a sequence in $L^2((0, T) \times \Gamma_\eps)$ such that
\begin{equation}
\eps \int_0^T \int_{\Gamma_\eps} |v^\eps(t, x)|^2 \, dt \, d\sigma_\eps(x) \le C
\end{equation}
where $C$ is a positive constant, independent of $\eps$. There exist a subsequence (still denoted by $\eps$) and a two-scale limit $v_0(t, x, y) \in L^2((0, T) \times \Omega; L^2(\Gamma))$ such that $v^\eps(t, x)$ two-scale converges to $v_0(t, x, y)$ in the sense that
\begin{equation}
\lim_{\eps \to 0} \eps \int_0^T \int_{\Gamma_\eps} v^\eps(t, x) \, \phi\left(t, x, \frac{x}{\eps}\right) dt \, d\sigma_\eps(x) = \int_0^T \int_\Omega \int_\Gamma v_0(t, x, y) \, \phi(t, x, y) \, dt \, dx \, d\sigma(y) 
\end{equation}
for any function $\phi \in C^1([0, T] \times \overline{\Omega}; C^\infty_\#(Y))$.
\end{prop}

\section{Well-prepared Initial Condition}\label{sec: appendix well prepared initial condition}

Our analysis is based on the fact that the solution remains nonnegative throughout the time and space. Furthermore, the initial condition should satisfy the compatibility condition as described in \eqref{compatibility condition} and the uniform (in $\eps$) estimates as described in \eqref{CI}. To ensure the solutions remains nonnegative, we consider positive initial condition satisfying the compatibility condition \eqref{compatibility condition}. Then the quasipositive \eqref{quasipositive} structure of the source term along with the positive Neumann boundary condition will ensure that our solution remains nonnegative throughout the time \cite[Lemma 1.1]{pierre2010global}. In this section, we focus on selecting the nonnegative initial condition for each domain $\Omega_{r_{\eps}}$, satisfying the compatibility condition \eqref{compatibility condition} and the uniform estimates \eqref{CI}. For $i=1,2,3$, consider smooth functions $a_{0}^i:\overline{\Omega}\to(0,\infty)$ such that
\[
\nabla a_{0}^i\cdot n=0 \ \text{at}\ \partial\Omega, \ i=1,2,3.
\]
Hence there exists a constant $\mathcal{I}_{min}>0$ such that
\begin{align}\label{minimum of initial condition}
    \xi a_{0}^i(x)\geq \mathcal{I}_{min}, \ \ \forall x\in\Omega \ \text{and} \ \forall\, i=1,2,3,
\end{align}
where $\xi =|Y^*|^{-1}$, where \(|Y^*|\) denotes the measure of the portion of the unit cube outside the hole, as defined in \eqref{area outside the hole in unit cube}.

Next we construct a well prepared initial condition from the smooth functions $a_0^i$ for $i=1,2,3$. We construct our initial conditions for periodically distributed spherical hole. It can be generalized to any type of holes with smooth boundary.

\begin{figure}[htbp]
    \centering
    \begin{tikzpicture}[scale=3]
  
   \draw[thick, fill=green!30] (-0.75, -0.75) rectangle (0.75,0.75);

     \draw[fill=white, thin] (0,0) circle (0.25);
	\draw[black, thick] (0, 0) -- (0.25, 0) 
       node[midway, above, sloped, inner sep=1pt] {\tiny $r_{\varepsilon}$};
	
	\draw[thick, dotted] (0,0) circle (0.5);
	\draw[black, thick] (-0.5, 0) -- (0, 0) 
        node at (-0.37,0.05) {\tiny $2r_{\varepsilon}$};
	\fill[black] (0,0) circle (0.02); 
    \node[below] at (0, -0.05) {\small $x_k$};
    
    \draw[thick, decorate, decoration={brace, amplitude=10pt, mirror}] 
        (-0.75, -0.8) -- (0.75, -0.8) 
        node[midway, below=12pt] {$\varepsilon$};

\end{tikzpicture} 
    \label{fig:unit cell spherical}
    \caption{$\eps$- cell}
\end{figure}

Let $r_{\eps} = \mathcal{\eps}$. Fix an $\eps$-cell inside $\Omega_{r_\eps}$ and let the spherical hole inside it is centered at $x_{k}$ with radius $r_\varepsilon<\frac{\eps}{4}$, which is  strictly contained inside $\Omega$. Let $k \in \mathcal{I}_{\varepsilon}$, where $\#\mathcal{I}_{\eps}$ denotes the number of holes at each $\varepsilon$-scale (recall \eqref{index of holes}). Consider the following equation: for $i=1,2,3$
\begin{equation}\label{epsilon hole equation, domain}
    \left \{
    \begin{aligned}
        -\Delta w^i_{k,\eps}=& 0, \qquad &&  \text{in}\ B(x_k, 2r_{\eps})\setminus \overline{B(x_k, r_{\eps})}\\
        \nabla w^i_{k,\eps}\cdot n= &-\xi\nabla a^i_0 \cdot n \qquad && \text{on}\ \partial B(x_k, r_{\eps})\\
        w_{k,\eps}^i=&0 \qquad && \text{on}\ \partial B(x_k, 2r_{\eps}).
    \end{aligned}
    \right .
\end{equation}
 Thanks to the fact that $\nabla a_0^i\cdot n\in \mathrm{L}^{\infty}(\partial B(x_k, r_{\eps}))$, hence in any $\mathrm{L}^p(\partial B(x_k, r_{\eps}))$ (for $1<p<\infty$), we can say that there exists unique solution $w^i_{k,\eps}$, in the following function space (see \cite[Theorem 3.28 and Lemma 3.22]{troianiello2013elliptic} for further details)
 \begin{align} \label{regularity, auxiliary eqution for initial condition}
 w^{i}_{k,\eps}\in \mathrm{L}^{\infty}\left( B(x_k, 2r_{\eps})\setminus \overline{B(x_k, r_{\eps})}\right)\cap \mathrm{W}^{2,p} \left( B(x_k, 2r_{\eps})\setminus \overline{B(x_k, r_{\eps})}\right).
 \end{align}
Let us define the following function
\begin{align}\label{well prepared initial condition}
    a_{\eps,0}^i:= \xi a_0^i \chi_{\Omega_{r_\eps}}+\sum\limits_{k\in\mathcal{I}_{\eps}} \hat{w}^1_{k,\eps}, \ \forall \, i=1,2,3,
\end{align}
where $\hat{w}^i_{k,\eps}$ is the extension of $w_{k,\eps}^i$ by zero outside $B(x_k, 2r_{\eps})$ and $\chi_{\Omega_{r_{\eps}}}$ denotes the characteristic function of $\Omega_{r_{\eps}}$. We establish that $a^i_{\eps,0}$ is a well prepared initial condition for $i=1,2,3$. Observe that
\[
\supp(\hat{w}^i_{k_1,\eps}) \cap \supp(\hat{w}^i_{k_2,\eps})=\Phi, \ \ \text{if}\ k_1\neq k_2,
\]
where $\Phi$ is the empty set. From the fact above, it is evident that the compatibility conditions is satisfied i.e.,
\begin{align}\label{compatibility in the holed boudary}
    \nabla a_{\eps,0}^i\cdot n=0 \ \  \text{on} \ \partial\Omega\cup \Big\{\cup_{k\in\mathcal{I}_{\eps}} \partial B(x_k, r_{\eps})\Big\}=\partial\Omega_{r_{\eps}}.
\end{align}
Our next goal is to show that for $\eps \ll 1$, the function $a_{\eps,0}^i$ remains nonnegative for $i=1,2,3$. The following Lemma help us guarantee the nonnegativity of the initial condition inside $\Omega
_{\eps}$.
\begin{lemma}\label{local cube equation in epsilon cube}
    Let $w^i_{k,\eps}$ be the solution of equation  \eqref{epsilon hole equation, domain} for $i=1,2,3$ and $k\in\mathcal{I}_{\eps}$. Then the following holds: for $i=1,2,3$
    \[
    | w_{k,\eps}^i|\leq \texttt{W} \, r_{\eps},  
    \]
    where $\texttt{W}>0$ is a constant that does not depend on $\eps$ and $k$.
\end{lemma}
\begin{proof}
    Consider the following change of variable
    \[
    x \leftrightarrow x_{k}+r_{\eps} y.
    \]
    Hence we can rewrite the equation \eqref{epsilon hole equation, domain} as: for $i=1,2,3$
    \begin{equation}\label{epsilon hole equation, domain, after transformation of variable}
    \left \{
    \begin{aligned}
         -\Delta_y w^i_{k,\eps}&= 0, \qquad &&  \text{in}\ B(0, 2)\setminus \overline{B(0,1)}=: \Lambda_0
        \\
         \nabla_y w^i_{k,\eps}\cdot n &= -\xi r_{\eps}\nabla a^i_0 \cdot n \qquad && \text{on}\ \partial B(0,1)=: \Lambda_0^1
        \\
         w_{k,\eps}^i &=0 \qquad && \text{on}\ \partial B(0, 2)=:\Lambda_0^2.
    \end{aligned}
    \right .
\end{equation}
The following regularity estimate holds for solutions to \eqref{epsilon hole equation, domain, after transformation of variable} (see \cite[Theorem 3.28]{troianiello2013elliptic})
\begin{align*}
    \| w_{k,\eps}^i\|_{\mathrm{W}^{2,p}(\Lambda_0)} \leq C_{\Lambda_0} \xi r_{\eps} \left\| -\nabla a^i_{\eps,0} \cdot n\right\|_{\mathrm{W}^{\frac 1p, p}(\Lambda_0^1)}, \ \ \forall \, i=1,2,3,
\end{align*}
where $ 1 < p \le \infty$ and $C_{\Lambda_0}$ is a positive constant. By the Sobolev Embedding $ \mathrm{W}^{1, p}(\Lambda_0^1)\hookrightarrow \mathrm{W}^{\frac 1p, p}(\Lambda_0^1)$ it follows that 
\begin{align*}
    \| w_{k,\eps}^i\|_{\mathrm{W}^{2,p}(\Lambda_0)} \leq C_{\Lambda_0} \xi C_{sob,1/p,1} r_{\eps} \left\| -\nabla a^i_{\eps,0} \cdot n\right\|_{\mathrm{W}^{1, p}(\Lambda_0^1)}, \ \ \forall \, i=1,2,3,
\end{align*}
where  $C_{sob,1/p,1}$ denotes the constant associated with the embedding $ \mathrm{W}^{1, p}(\Lambda_0^1)\hookrightarrow \mathrm{W}^{\frac 1p, p}(\Lambda_0^1)$.
\\
Since $\| a^i_{0}\|_{\mathrm C^2(\overline{\Omega})}\leq C_{max,in}$ we have
\[
\|w^{i}_{k,\eps}\|_{\mathrm{W}^{2,p}(\Lambda_0)} \leq 2^{\frac 1p}C_{\Lambda_0} \xi C_{sob,1/p,1} C_{max,in} | \Lambda_0^1|^{\frac{1}{p}} r_{\eps}.  
\]
Recall the following Sobolev Embedding
\begin{equation}\label{sobolev embedding w2p infty}
    \mathrm{W}^{2,p}(\Lambda_0) \hookrightarrow \mathrm{L}^{\infty}(\Lambda_0),
\end{equation}
and let $C_{sob, p,\infty}$ denote the constant that appears in the above embedding. Define 
\[
\texttt{W}:= 2^{\frac 1p} \xi C_{\Lambda_0} C_{sob,1/p,1} C_{max,in} | \Lambda_0^1|^{\frac{1}{p}}C_{sob, p,\infty}.
\]
Invoking \eqref{sobolev embedding w2p infty} yields the desired estimate.
\end{proof}

\begin{corollary} \label{positivity of initial condition at epsilon scale}
    For $1=1,2,3$, let $a^i_{\eps,0}$ be as defined in \eqref{well prepared initial condition}. Then for $\eps\ll 1$, $a^i_{\eps,0}$ always remains nonnegative.
\end{corollary}
\begin{proof}
    The proof follows from the fact that the support of each $w^i_{k,\eps}$ is disjoint for a fixed $i$ and contained in $\Omega^{\delta}$ as described in \eqref{security zone}. Let $x\in \Omega\setminus\Omega^{\delta}$ or $x\not\in$\, $\supp (w_{k,\eps}^i)$ for all $k\in\mathcal{I}_{\eps}$. It implies $a^i_{\eps,0}\geq 0$. Let $x\in \Omega^{\delta}$ and there exists $k\in\mathcal{I}_{\eps}$ such that $x\in$\, $\supp (w_{k,\eps}^i)$. Observe that there exists unique $k=k_x\in \mathcal{I}_{\eps}$ such that $x\in$\, $\supp (w_{k_x,\eps}^i)$. Hence
    \begin{align*}
        a^i_{\eps,0} (x) = \xi a^i_{0} (x) + w^i_{k_x,\eps}(x) \geq  \mathcal{I}_{\min}- \texttt{W} r_{\eps}, \ x\in \Omega^{\delta}
    \end{align*}
   where we have used the inequality \eqref{minimum of initial condition} and Lemma \ref{local cube equation in epsilon cube} for the inequality. Choosing $\eps\ll 1$ yields the desired result.  
\end{proof}

In the next lemma we derive uniform integrability estimate of the well prepared initial condition \eqref{well prepared initial condition} with respect to $\eps$.
\begin{lemma}\label{uniform H1 estimate initial condition}
     For $i=1,2,3$, let $w^i_{k,\eps}$ satisfies equation  \eqref{epsilon hole equation, domain} and $\hat{w}^i_{k,\eps}$ be it's extension to $\Omega_{r_{\eps}}$ by zero. Then, there exists a constant $C_w>0$, independent of $\eps$ and $k$, such that
    \begin{align*}
       \left\|\sum_{k\in\mathcal{I}_{\eps}} \hat{w}^i_{k,\eps}\right\|_{\mathrm{L}^2(\Omega_{r_{\eps}})} \leq & C_{W}  r_{\eps},\\
    \left\|\sum_{k\in\mathcal{I}_{\eps}} \hat{w}^i_{k,\eps}\right\|_{\mathrm{ H}^1(\Omega_{r_{\eps}})} \leq & C_{W},\\
    \left\|\sum_{k\in\mathcal{I}_{\eps}} \hat{w}^i_{k,\eps}\right\|_{\mathrm{L}^4(\Omega_{r_{\eps}})} \leq & C_{W} r_{\eps}.
    \end{align*}
\end{lemma}
\begin{proof}
    We will revisit the equation as described in \eqref{epsilon hole equation, domain, after transformation of variable}. The change of variable $x \leftrightarrow x_{k}+r_{\eps} y$ yields
      \begin{equation}\label{epsilon hole equation, domain, after transformation of variable, revisit}
    \left \{
    \begin{aligned}
        -\Delta_y w^i_{k,\eps}=& 0, \qquad &&  \text{in}\ B(0, 2)\setminus \overline{B(0,1)}=: \Lambda_0\\
        \nabla_y w^i_{k,\eps}\cdot n= &- \xi r_{\eps}\nabla a^i_0 \cdot n \qquad && \text{on}\ \partial B(0,1)=: \Lambda_0^1\\
        w_{k,\eps}^i=&0 \qquad && \text{on}\ \partial B(0, 2)=:\Lambda_0^2.
    \end{aligned}
    \right .
\end{equation}
Testing this equation with $w^i_{k,\eps}$ yields: for $i=1,2,3$
\begin{align}
    \int_{\Lambda_0} | \nabla_y w^i_{k,\eps}|^2 \, dy=& \xi \int_{\Lambda_0^1} r_{\eps} (-\nabla a^i_{0}\cdot n) w^i_{k,\eps} \, d\sigma_y \nonumber \\
    \leq & \xi \int_{\Lambda_0^1} r_{\eps} \| \nabla a^i_{0} \|_{\mathrm{L}^{\infty}(\overline{\Omega})} |w^i_{k,\eps}| \, d\sigma_y \nonumber\\
    \leq & \frac 12 C_1 \xi^2 r^2(\eps) \| \nabla a^i_{0} \|_{\mathrm{L}^{\infty}(\overline{\Omega})}^2 | \Lambda_0^1| + \frac{1}{2C_1} \int_{\Lambda_0^1} |w^i_{k,\eps}|^2\, d\sigma_y, \label{step before poincare, initil condition}
\end{align}
where $C_1$ is a positive constant to be chosen later. We use the following trace embedding
\[
\| f\|_{\mathrm{L}^2(\Lambda_0^1\cup \Lambda_0^2)} \leq C_{trace} \|f\|_{\mathrm{ H}^1(\Lambda_0)}, \ \forall \, f\in \mathrm{ H}^1(\Lambda_0), \ C_{trace}>0,
\]
in \eqref{step before poincare, initil condition}. This leads to
\begin{align}\label{step 2 before poincare, initil condition}
    \int_{\Lambda_0} | \nabla_y w^i_{k,\eps}|^2 \, dy \leq \frac 12 C_1 \xi^2 r^2(\eps) \| \nabla a^i_{0} \|_{\mathrm{L}^{\infty}(\overline{\Omega})}^2 | \Lambda_0^1| + \frac{C_{trace}^2}{2C_1} \| w^i_{k,\eps}\|_{\mathrm{ H}^1(\Lambda_0)}.
\end{align}
Owing to the fact that $|\Lambda_0^2|\neq 0$, we obtain the following Poincar\'e inequality
\[
C_{p}\int_{\Lambda_0}|w^i_{k,\eps}|^2 \, dy \leq \int_{\Lambda_0} | \nabla_y w^i_{k,\eps}|^2 \, dy.
\]
Using this in \eqref{step 2 before poincare, initil condition}, we obtain the following relation
\begin{align*}\label{step 3 before poincare, initil condition}
    \left( \frac 12+ \frac{C_p}{2}\right) \| w^i_{k,\eps}\|_{\mathrm{ H}^1(\Lambda_0)}^2 \leq C_2 r^2_\eps+ \frac{C_{trace}^2}{2C_1} \| w^i_{k,\eps}\|_{\mathrm{ H}^1(\Lambda_0)},
\end{align*}
where $C_2:= \frac 12 \xi^2 C_1 \max_{i=1,2,3}\left\{\| \nabla a^i_{0} \|_{\mathrm{L}^{\infty}(\overline{\Omega})}^2\right\} | \Lambda_0^1|$. Let us choose $C_1$ such that $\frac{C_{trace}^2}{C_1}\times\frac{1}{C_p+1}\leq \frac 12$. It leads us to the following bound
\begin{align}\label{H1 estimate around origin, initial condition}
    \| w^i_{k,\eps}\|_{\mathrm{ H}^1(\Lambda_0)}^2 \leq C_3 r^2_\eps,
\end{align}
where $C_3:=\frac{4C_2}{1+C_p}$. The above estimate can be expressed as
\begin{equation}\label{auxiliary relation to obtain L4, initial condition}
\left \{
\begin{aligned}
    \int_{\Lambda_0} |w^i_{k,\eps}|^2 \, dy \leq & C_3 r^2_\eps,\\
    \int_{\Lambda_0} |\nabla_y w^i_{k,\eps}|^2 \, dy \leq & C_3 r^2_\eps.
    \end{aligned}
    \right .
\end{equation}
We replace $x\leftrightarrow x_k+r_{\eps}y$. It yields
\begin{align*}
    \int\limits_{B(x_k, 2r_{\eps})\setminus\overline{B(x_k, r_{\eps})}}| w^i_{k,\eps}|^2 r_{\eps}^{-3}\, dx  \leq & C_3 r^2_\eps,\\
    \int\limits_{B(x_k, 2r_{\eps})\setminus\overline{B(x_k, r_{\eps})}}|\nabla_x w^i_{k,\eps}|^2 r_{\eps}^{-1}\, dx  \leq & C_3 r^2_\eps.
\end{align*}
The above relation can be rewritten as
\begin{equation} \label{H1 estimate around arbitrary centre, initial condition}
\left \{
\begin{aligned}
     \int\limits_{B(x_k, 2r_{\eps})\setminus\overline{B(x_k, r_{\eps})}}| w^i_{k,\eps}|^2\, dx  \leq & C_3 r^5_\eps,\\
    \int\limits_{B(x_k, 2r_{\eps})\setminus\overline{B(x_k, r_{\eps})}}|\nabla_x w^i_{k,\eps}|^2\, dx  \leq & C_3 r^3_\eps.
    \end{aligned}
    \right .
    \end{equation}
Summing them over all holes, we obtain
\begin{equation} \label{total sum H1 estimate around arbitrary centre, initial condition}
\left \{
\begin{aligned}
     \sum_{k\in\mathcal{I}_{\eps}}\int\limits_{B(x_k, 2r_{\eps})\setminus\overline{B(x_k, r_{\eps})}}| w^i_{k,\eps}|^2\, dx  \leq & C_3 r^5_\eps \# \mathcal{I}_{\eps},\\
    \sum_{k\in\mathcal{I}_{\eps}}\int\limits_{B(x_k, 2r_{\eps})\setminus\overline{B(x_k, r_{\eps})}}|\nabla_x w^i_{k,\eps}|^2\, dx  \leq & C_3 r^3_\eps \# \mathcal{I}_{\eps}.
\end{aligned}
\right .
\end{equation}
We use the fact that $r^3_\eps\#\mathcal{I}_{\eps}\leq \gamma_1$ \eqref{cond: boundary Gammaeps}. Furthermore, thanks to the fact that the support of each $w^i_{k,\eps}$ is disjoint (for a fixed $i$) and they can be extended as $\mathrm{ H}^1$ functions outside $B(x_k, 2r_{\eps})$ by zero, we have that
\begin{equation} \label{total sum H1 estimate around arbitrary centre, initial condition extended}
\left \{
\begin{aligned}
     \left\|\sum_{k\in\mathcal{I}_{\eps}} \hat{w}^i_{k,\eps}\right\|_{\mathrm{L}^2(\Omega_{r_{\eps}})}^2 \leq & C_3 \gamma_1 r^2_\eps,\\
    \left\|\sum_{k\in\mathcal{I}_{\eps}} \hat{w}^i_{k,\eps}\right\|_{\mathrm{ H}^1(\Omega_{r_{\eps}})}^2 \leq & C_3 \gamma_1,
\end{aligned}
\right .
\end{equation}
where $\hat{w}^i_{k,\eps}$ is the extension of $w^i_{k,\eps}$ by zero beyond $B(x_k, 2r_{\eps})$. We use the following Sobolev embedding
\[
\| f\|_{\mathrm{L}^4(\Lambda_0)} \leq C_{sob} \| f\|_{\mathrm{ H}^1(\Lambda_0)}, \ \forall\, f\in \mathrm{ H}^1(\Lambda_0), \ C_{sob}>0,
\]
in \eqref{auxiliary relation to obtain L4, initial condition}. It yields
\[
\| w^i_{k,\eps}\|_{\mathrm{L}^4(\Lambda_0)}^4 \leq C_3^2 C_{sob}^2 r^4_\eps.
\]
Again, change of variable $x\leftrightarrow x_k+r_{\eps}y$ yields
\[
\| w^i_{k,\eps}\|_{\mathrm{L}^4\left(B(x_k, 2r_{\eps})\setminus\overline{B(x_k, r_{\eps})}\right)}^4 \leq C_3^2 C_{sob}^2 r^7_\eps.
\]
Using the fact that the support of the functions $w^i_{k,\eps}$ are disjoint (for a fixed $i$), we have that
\[
\left\|\sum_{k\in\mathcal{I}_{\eps}} \hat{w}^i_{k,\eps}\right\|_{\mathrm{L}^4(\Omega_{r_{\eps}})}^4 \leq C_3^2 C_{sob}^2 \gamma_1 r^4_\eps,
\]
where $\gamma_1$ is as defined in \eqref{cond: boundary Gammaeps}. Choice of the constant $C_{W}:= \max\{\sqrt{C_{3}\gamma_1}, \sqrt{C_3 C_{sob}}\gamma_1^{\frac 14}\}$, yields our result.
\end{proof}
The above $\mathrm{L}^2$ integral estimate helps us to obtain two-scale limit (see Appendix~\ref{sec: appendix two-scale}) of the well prepared initial data $a^i_{\eps,0}$, for each $i=1,2,3$.
\begin{lemma}\label{two scale limit of initial condition}
  For $i=1,2,3$, consider the well prepared initial condition $a^i_{\eps,0}$ as described in \eqref{well prepared initial condition}. Then, as $\varepsilon \to 0$, the two-scale limit $($see Appendix~\ref{sec: appendix two-scale}$)$ is given by
\[
a^i_{\varepsilon,0} \xrightharpoonup[]{2\text{-scale}} \xi a_0^i \chi_{Y^*}.
\]
\end{lemma}
\begin{proof}
    Thanks to Lemma \ref{uniform H1 estimate initial condition}, we have that
    \[
    \left\|\sum_{k\in\mathcal{I}_{\eps}} \hat{w}^i_{k,\eps}\right\|_{\mathrm{L}^2(\Omega_{r_{\eps}})} \to 0, \ \text{as}\ \eps\to 0.
    \]
    Hence, the 2-scale limit of $\displaystyle{\sum_{k\in\mathcal{I}_{\eps}} \hat{w}^i_{k,\eps}}$ is $0$. Furthermore, since
    \[
    \alpha a^i_0 \chi_{\Omega_{r_{\eps}}} \xrightharpoonup[]{2-scale} \xi a_0^i \chi_{Y^*},
    \]
    we conclude our result.
    \end{proof}

\begin{remark}\label{remark 5}
When $r_{\eps} = \mathcal{O}(\eps^{\alpha})$ and $\alpha >1$ we observe that 
\[
\chi_{\Omega_{r_{\eps}}} \to 1 \text{ strongly in } \mathrm L^2(\Omega).
\]
Hence a similar construction as done for the case of $r_{\eps} = \mathcal{O}(\eps)$, will produce a sequence of initial data $\{a^i_{\varepsilon,0} \}_{\eps}$ which strongly converges to $a^i_{0}$ in $\mathrm L^2(\Omega)$ and satisfies the compatibility condition \eqref{compatibility condition} and the uniform estimates \eqref{CI}.  
\end{remark}


 	\vspace{.3cm}
	\textbf{Data Availability:} Data sharing is not applicable to this article as no datasets were generated or analyzed during the current study.

	\vspace{.3cm}
	\textbf{Conflict of Interest:} The authors declare that there are no conflicts of interest relevant to the content of this manuscript.

\bibliographystyle{plain}
\begin{singlespace}
\bibliography{biblio}
\end{singlespace}
\end{document}